\documentclass[a4paper]{amsart}
\usepackage{amsmath,amsthm,amssymb,latexsym,epic,bbm,comment,mathrsfs}
\usepackage{graphicx,enumerate,stmaryrd, xcolor, color}
\usepackage[all,2cell]{xy}
\xyoption{2cell}

\usepackage{soul}

\theoremstyle{plain}
\newtheorem{thm}{Theorem}
\newtheorem*{thm*}{Theorem}

\newtheorem{lem}[thm]{Lemma}
\newtheorem{prop}[thm]{Proposition}

\newtheorem{df-prop}[thm]{Definition-Proposition}

\theoremstyle{definition}

\theoremstyle{remark}
\newtheorem{rem}[thm]{Remark}
\newtheorem{ex}[thm]{Example}

\usepackage[all]{xy}
\usepackage[active]{srcltx}
\usepackage[parfill]{parskip}
\usepackage{enumerate}

\newcommand{\Hom}{\operatorname{Hom}}

\usepackage{hyperref}
\newcommand{\mc}{\mathcal}
\newcommand{\mf}{\mathfrak}
\newcommand{\C}{\mathbb C}

\newcommand{\oa}{{\bar 0}}
\newcommand{\ob}{{\bar 1}}
\def\gl{\mathfrak{gl}}
\newcommand{\g}{\mathfrak{g}}

\def\la{\lambda}

\def\ov{\overline}
\newcommand{\ch}{\mathrm{ch}}
\newcommand{\h}{\mathfrak{h}}
\newcommand{\N}{\mathbb{Z}_{\geq 0}}
\newcommand{\Z}{{\mathbb Z}}

\def\Mod{\operatorname{-Mod}\nolimits}

\newcommand{\hf}{{\Small \frac12}}

\def\Hom{\operatorname{Hom}\nolimits}
\def\End{\operatorname{End}\nolimits}
\def\Res{\operatorname{Res}\nolimits}
\def\Ind{\operatorname{Ind}\nolimits}

\def\gl{\mathfrak{gl}}

\def\la{\lambda}

\def\ov{\overline}
\newcommand{\ad}{\mathrm{ad}}

\begin{document}
\title[Whittaker Modules and Finite $W$-Algebras of Queer Lie Superalgebras]{Whittaker Modules and Representations of Finite $W$-Algebras of Queer Lie Superalgebras}

\author[Chen]{Chih-Whi Chen} \address{Department of Mathematics, National Central University, Chung-Li, Taiwan 32054 \\ National Center of Theoretical Sciences,
	Taipei, Taiwan 10617} \email{cwchen@math.ncu.edu.tw}

\author[Cheng]{Shun-Jen Cheng}
\address{Institute of Mathematics, Academia Sinica, Taipei, Taiwan 10617} \email{chengsj@math.sinica.edu.tw}

\begin{abstract}
 We study various categories of Whittaker modules over the queer Lie superalgebras $\mf q(n)$. We formulate standard Whittaker modules and reduce the problem of composition factors of these standard Whittaker modules to that of Verma modules in the BGG categories $\mc O$ of $\mf q(n)$. We also obtain an analogue of Losev-Shu-Xiao decomposition for the finite $W$-superalgebras $U(\mf q(n), E)$ of $\mf q(n)$ associated to an odd  nilpotent element $E\in \mf q(n)_\ob$. As an application, we establish several equivalences of categories of Whittaker $\mf q(n)$-modules and analogues of BGG category of $U(\mf q(n), E)$-modules. In particular, we reduce the multiplicity problem of Verma modules over $U(\mf q(n), E)$ to that of the Verma modules in the BGG categories $\mc O$ of $\mf q(n)$.
 \end{abstract}

\maketitle

\tableofcontents



\section{Introduction}

\subsection{Background}
	The representation theory of Lie superalgebras has developed into an important branch in Lie and representation theory. Over the past few decades, there has been significant progress in the study of the Bernstein-Gelfand-Gelfand (BGG) category $\mc O$, which has yielded satisfactory development and broad applications. Significantly, this recent advancement has also established deep connections to related fields and provided a robust framework that serves as a model for the study of more general representations.  

    The investigation into representations beyond the weight modules, often inspired by the structures observed in category $\mc O$, has recently brought non-weight modules to the forefront of research. Among them, {\em Whittaker modules}, originally pioneered by Kostant in \cite{Ko78} in the setup of reductive Lie algebras $\g$, are one of the most  interesting objects of exploration; see also \cite{AB21, B97,MS97, MS14}. These Whittaker $\g$-modules, which we refer to as MMS Whittaker modules \cite{Mc85,MS97} and which we denote by $\mc N(\zeta)$, where $\zeta$ is a nilpotent character of $\g$, are generalizations of the BGG categories, which are recovered as the special case $\zeta = 0$.  In \cite{Ko78}, Kostant essentially treated another extreme case of {\em non-singular} $\zeta$, namely the case when the dimension of the kernel $\ker \zeta$ is minimal and gave a   classification of the simple Whittaker modules in $\mc N(\zeta)$ for this situation.

    A natural extension of this theory is the theory of Whittaker modules for {\em basic classical}, {\em periplectic} and {\em queer} Lie superalgebras from Kac's list \cite{K77}.    	Although considerable progress has been made in the study of Whittaker modules over Lie algebras, the corresponding theory for Lie superalgebras has only recently begun to receive increased attention; see, e.g., \cite{BCW14,C21,CCM23,CC23_2,CC22}. However, within the framework of Lie superalgebras, the existing literature primarily focuses on Whittaker modules over basic classical and periplectic Lie superalgebras, while the study of those associated with the queer series remains less developed.

 Within the same time as the development of Whittaker modules, insights stemming from Kostant’s work \cite{Ko78} led to a substantial interest in, and to a rapid development of representation theory of finite $W$-algebras, which by now, is a significant research area in representation theory. Following the general construction of Premet \cite{Pr02}, a finite $W$-algebra $U(\g, e)$ is an associative superalgebra constructed from a pair $(\g, e)$, where $\g$ is a finite-dimensional Lie superalgebra and $e$ is a (homogeneous) nilpotent element of $\g$; see also \cite{Pr07, PS16, Wa11, Zh14} and reference therein.  The rapid development of this field is partially due to the discovery of connections between the representation theory of such finite $W$-algebras with other areas  such as representations of Lie (super)algebras, the primitive ideals for the universal enveloping algebras, Yangians and affine $W$-algebras etc., see, e.g., \cite{BK06, BK08,Pr07,RS99,Zhu96} and references therein. 
Among these numerous connections, perhaps the most straightforward one is to Whittaker modules over Lie superalgebra, a connection originally proved by Skryabin \cite{Skr} for Lie algebras, known as the {\em Skryabin equivalence}, which establishes an equivalence of the category of Whittaker modules for a Lie superalgebra $\g$ and the category of modules over the finite $W$-algebra $U(\g, e)$; see also \cite{Los10b,Zh14,ZS15}.

  The initial formulation of the generalization of the BGG category $\mathcal{O}$ to finite $W$-algebras associated to semisimple Lie algebras, including the definition of Verma and highest weight modules, was first established and studied in \cite{BGK08}. We shall also call these categories BGK category of finite $W$-algebras. In \cite{Los12}, Losev made a remarkable extension of the Skryabin’s equivalence connecting the BGK category $\mathcal{O}$ of finite $W$-algebras with MMS Whittaker categories.  
  Notably, this latter result, together with \cite{B97,MS14}, implies that the composition multiplicities of Verma modules in the category $\mathcal{O}$ of the finite $W$-algebra of reductive Lie algebras are computed by the Kazhdan–Lusztig polynomials. Such a relationship was originally conjectured in \cite{BGK08}.  To prove this equivalence of categories Losev established a fundamental decomposition theorem on a certain completion $U(\g)^\wedge_{\widetilde{\mf m}}$ with respect to $\widetilde{\mf m}$ in \eqref{eq::23::m} based on Fedosov quantization; see also \cite{Los10b}.
 
 For Lie superalgebras, Shu and Xiao \cite{SX20} developed an algebraic approach to give a super generalization of Losev's construction \cite{Los10b} to the case where $\g$ is a basic classical Lie superalgebra. More recently, the formulation of the category $\mc O$ of finite $W$-algebras in the sense of \cite{BGK08} was subsequently extended to basic Lie superalgebras in \cite[Section 5.2]{CW25}. In this context, Wang and the second author applied the Losev's decomposition, as adapted by Shu–Xiao to Lie superalgebras, to construct equivalences between the corresponding categories of MMS Whittaker $\g$-modules and $U(\g, e)$-modules, for a basic classical Lie superalgebra $\g$.


\subsection{The goal} The main goal of this paper is to extend the theory of Whittaker modules, finite $W$-algebras, and the Losev–Shu–Xiao decomposition from the basic classical case to the setting of queer Lie superalgebras ${\mf q}(n)$. It is worth pointing out that the development of Whittaker module theory and the finite $W$-algebras for queer Lie superalgebras presents distinct structural challenges in comparison to the basic classical Lie superalgebras, which we shall explain as follows.

\subsubsection{MMS Whittaker modules} 
The MMS category $\mc N(\zeta)$ of Whittaker $\g$-modules associated with a nilpotent character $\zeta$  was initially introduced by McDowell \cite{Mc85} and Mili{\v{c}}i{\'c}-Soergel \cite{MS97} in the setup of reductive Lie algebras. In loc.~cit., the most basic among the Whittaker modules over Lie algebras is the so-called {\em standard Whittaker module}, denote by $M(\la,\zeta)$, where $\la$ is a weight. They are parabolically induced modules from Kostant's simple Whittaker modules over a certain Levi subalgebra $\mf l_\zeta$ and can be viewed as generalizations of Verma modules, as $M(\la,0)$ reduces to the Verma module of highest weight $\la$. Furthermore, the simple tops of these modules constitute a complete list of simple objects in the category $\mc N(\zeta)$. Specifically, it was shown in \cite{MS97} that the composition multiplicities for   $M(\lambda, \zeta)$ associated to an integral weight $\lambda$ are determined by Kazhdan–Lusztig polynomials, just
like for Verma modules in category $\mc O$. The corresponding full generalization for all weights was later obtained by Backelin \cite{B97}, where the proof relies on a certain exact functor $\Gamma_{\zeta}$ from  $\mc O$ to $\mc N(\zeta)$. Recent work adapts these constructions and results to basic and periplectic Lie superalgebras $\g$, under the assumption that the $\mf l_\zeta$ is a purely even Levi subalgebra of $\g$; see, e.g., \cite{C21,CC22}. However, in the case of ${\mf q}(n)$, any Levi subalgebra arising from a parabolic decomposition cannot be a purely even. Consequently, Kostant's simple Whittaker modules cannot be parabolically induced to ${\mf q}(n)$-modules.

 In the present paper, we work with the standard triangular decomposition $\mf{g} =\mf n\oplus \mf h\oplus \mf n^-$, defined in \eqref{eq::tri}, of the queer Lie superalgebra $\g= {\mf q}(n)$. We study the MMS type category $\mc N(\zeta)$ of Whittaker modules over $\mf q(n)$ associated with a  character $\zeta$ of $\mf n_\oa$. In particular, we give a formulation of standard Whittaker modules and reduce the problem of classification of simple objects in $\mc N(\zeta)$ to that of the simple objects in the counterpart of the MMS category of $\mf l_\zeta$ associated with a non-singular character of $\mf n_\oa$ (see Theorem \ref{cls::thm}). Furthermore, we achieve a reduction of the composition multiplicity problem of the standard Whittaker modules to that of Verma modules in the category $\mc O$, using an analogue of Backelin functor $\Gamma_\zeta$ for $\mf q(n)$ (see Theorem \ref{eq::compostadWhi}). Moreover, we study the target category $\mc W(\zeta):= \Gamma_\zeta(\mc O)$ of $\Gamma_\zeta$, which serves as a unified framework encompassing  all simple and standard Whittaker modules and is a generalization of the ordinary BGG category $\mc O$, as $\mc O$ coincides with the special case $\mc W(0)$. For this category, we prove a BGG-type reciprocity, generalizing the BGG reciprocity for the ordinary category $\mc O$ of $\mf q(n)$.

\subsubsection{Finite W-algebras}   

In the classical setting of reductive Lie algebras $\g$ and basic classical Lie superalgebras $\g$, the construction of the finite $W$-algebra $U(\mathfrak{g}, e)$ associated with an even nilpotent element $e \in \g_\oa$ relies  on 
an even non-degenerate invariant supersymmetric bilinear form $(\_|\_)$ on $\mathfrak{g}$. The queer Lie superalgebra $\mathfrak{q}(n)$ is known to lack such an even non-degenerate invariant bilinear form. 

In the present paper, we consider the finite $W$-algebras $U(\mf q(n), E)$ for $\mf q(n)$ associated to odd nilpotent elements $E\in \mf q(n)_{\ob}$, which have been studied in \cite{Zh14, PS16}.  Unlike the classical construction, these algebras are defined based on the odd non-degenerate invariant supersymmetric bilinear form inherent to $\mathfrak{q}(n)$. This reliance on the odd bilinear form is a distinguished feature of $\mf q(n)$, thereby setting its $W$-algebra construction apart from the more standard approaches.

In the paper, we study various BGG-type categories $\mc O$ of $U(\mf q(n), E)$-modules, where the construction of these categories is adopted from \cite{BGK08}. In particular, we prove a version of the Losev–Shu–Xiao decomposition for $\mf q(n)$ (see Theorem \ref{thm::Losevdecom}). Using the constructions in \cite{Los12}, we then establish several equivalences of categories of Whittaker $\g$-modules and the categories $\mc O$ of $U(\mf q(n), E)$-modules (see Theorem \ref{thm::equivWhiWmod}).  In particular, the MMS category $\mc N(\zeta)$ is equivalent to the full subcategory of objects of finite length in the category $\mc O$ of $U(\mf q(n), E)$. As a consequence, we reduce the problem of determining the composition multiplicities of Verma modules in the category $\mathcal{O}$ for the finite $W$-superalgebra for $\mf q(n)$   to those of Verma modules in the BGG category of $\mathfrak{q}(n)$.

\subsubsection{}\label{rem::generalization}  It is worth pointing out that most of the constructions and results in Section \ref{sect::2} on Whittaker modules for $\mf q(n)$ admit natural generalizations to the setting of basic classical Lie superalgebras $\g$ from Kac's list \cite{K77}. In such cases, the necessary ingredients required for some of these results can be obtained from the established results in \cite{CC23_2} and \cite{CC22}. We observe that in the series of articles \cite{C21,CCC21,CCM23,CC23_2,CC22}, the purely even subalgebra $\mf l_\zeta$ was always assumed to be a Levi subalgebra of $\g$. Consequently, the results presented in Section \ref{sect::2} regarding Whittaker modules can be adapted to analogous ones for basic classical Lie superalgebras, thus completing the remaining general cases. In particular, for an arbitrary basic classical Lie superalgebra, we can analogously define standard Whittaker modules in full generality and reduce the problem of computing  their composition multiplicities to that of Verma modules. This also recovers \cite[Theorem 1]{CC22}, which was stated in loc.~cit.~under the condition that the Lie superalgebra is basic classical and the  a Levi subalgebra of $\g$.

\subsection{Organization} This paper is organized as follows. In Section \ref{sect::pre} we provide the necessary preliminaries. In particular, we introduce various equivalent abelian categories, including the category of projectively presentable modules in $\mc O$, and establish the BGG-type reciprocity for these categories.  These results will be used in establishing the BGG-type reciprocity for $\mc W(\zeta)$ in the subsequent sections.  

In Section \ref{sect::2}, we study the MMS Whittaker category $\mc N(\zeta)$ for the queer Lie superalgebra. The reduction of the classification problem for simple objects in $\mathcal{N}(\zeta)$ is presented in Subsection~\ref{sect::clsofsimpleWhi}. We formulate in Subsection \ref{sect::bac} the standard Whittaker modules $M(\la,\zeta)$ and establish some properties of the Backelin functor $\Gamma_\zeta$ that we need for solving the multiplicity problem of $M(\la,\zeta)$. Subsection \ref{sec::CatW} is devoted to the study of the Whittaker category $\mc W(\zeta)$.

In Section \ref{sect::Walg}, we focus on the study of finite $W$-algebras of $\mf q(n)$.  We develop the Losev–Shu–Xiao decomposition in the queer Lie superalgebra setting, and establish several equivalences of categories of Whittaker $\mf q(n)$-modules and $U(\mf q(n), e)$-modules.

\vskip0.3cm
{\bf Acknowledgment.}   The authors are partially supported by National Science and Technology Council grants of the R.O.C., and they further acknowledge support from the National Center for Theoretical
Sciences. The second author thanks Weiqiang Wang for interesting discussion.

\vskip0.3cm
 \section{Preliminaries} \label{sect::pre}
In this section, we review the necessary background and definitions that will be used in the remainder
of the paper.

\subsection{The queer Lie superalgebra $\mf {q}(n)$} \label{sect::pre::21}
 Throughout, $n$ denotes a fixed positive integer.  Let   $\gl(n|n)$ be the general linear Lie superalgebra of linear transformations on the complex superspace  $\C^{n|n}$ of superdimension $(n|n)$.  We may realize $\gl(n|n)$ as $2n\times 2n$ complex matrices of the form \[ \mf{gl}(n|n)= \left\{ \left( \begin{array}{cc} A & B\\	C & D\\ \end{array} \right) \| ~A, B, C, D \in \C^{n\times n} \right\}. \]

The \emph{queer Lie superalgebra} $\g := \mf q(n)$ is the subalgebra of $\mf gl(n|n)$ consisting of matrices of the form: 
\begin{align}
\mf g:=
\mf q(n)=
\left\{ \left( \begin{array}{cc} A & B\\
	B & A\\
\end{array} \right) \| ~A,B\in \C^{n\times n} \right\} \label{eq::qnelts}
\end{align}  
The even part $\g_\oa$ is isormophic to  $\gl(n)$, and the odd part $\g_\ob$ is isomorphic to $\g_\oa$ as an adjoint $\g_\oa$-module. The center $\mf z(\g)$ of $\g$ is spanned by the identity matrix $\text{I}_{2n}\in \mf \g_\oa$.  For $1\leq i, j \leq n$, we write $e_{ij}$ and $f_{ij}$ for the respective elementary matrices in $\g_\oa$ and $\g_\ob$ with $1$ on the $(i,j)$-entry and $0$ otherwise. We denote by $\Pi: \g\rightarrow \g$ the parity-reversing map that interchanges $e_{ij}$ and $f_{ij}$, for all $1\leq i,j\leq n$.



Let $\mf b$ be the standard Borel subalgebra of $\g$, which consists of matrices of the form \eqref{eq::qnelts} with $A$ and $B$ upper triangular. We let  \begin{align} 	&\g =\mf n\oplus \mf h\oplus \mf n^- \label{eq::tri} \end{align}   be the corresponding triangular decomposition, where  $\mf h$ is the standard Cartan subalgebra of $\g$ and $\mf n$ is the nilradical of $\mf b$. The sets $\{h_i:=e_{ii}|~1\leq i\leq n\}$ and $\{\ov h_i:=f_{ii}|~1\leq i\leq n\}$ form bases for $\h_\oa$ and $\h_\ob$, respectively. Let $\{\epsilon_i|~1\leq i\leq n\}$ be the basis for $\h^*_\oa$ dual to $\{h_i|~1\leq i\leq n\}$. We define a symmetric bilinear form $(\_,\_)$ on $\h_\oa^\ast$ by $(\epsilon_i,\epsilon_j)=\delta_{ij}$, for $1\leq i,j\leq n$.  

Let $\Phi$, $\Phi^+$ and $\Phi^-$  be the set of all, positive and negative roots of $\g$ with respect to $\h_\oa$, respectively. We denote by $\Phi_\oa$ and $\Phi_\ob$ the set of even and odd roots, respectively. Then   $\Phi^-=-\Phi^+$ and $\Phi^+ = \Phi^+_\oa \sqcup \Phi^+_\ob$, where $\Phi^+_\oa =  \Phi^+_\ob =\{\epsilon_i-\epsilon_j|~1\leq i< j\leq n\}$ ignoring the parity. We have the corresponding root space decomposition $\g = \bigoplus_{\alpha \in \Phi\cup \{0\}}\g^\alpha$, where $\g^\alpha$ denotes the root space for $\alpha \in \Phi$. Furthermore, for any  $\h_\oa$-submodule $V$ of $\g$, we let  $\Phi(V)$ denote the set of weights appearing in $V$.

We consider the BGG category  $\mc O = \mc O(\g,\mf n)$    of finitely generated $\mf h_\oa$-semisimple $\g$-modules on which the action of $U(\mf n)$ is locally finite. This category $\mc O$ can be alternatively defined as the category of all $\g$-modules that restrict to the classical BGG category $\mc O(\mf \g_\oa)$ of $\g_\oa$-modules for the triangular decomposition $\g_\oa =\mf n_\oa \oplus \mf h_\oa\oplus {\mf n}^-_\oa$.  

The Weyl group $W$ is defined as the Weyl group of $\g_\oa$ and acts naturally on $\h_\oa^\ast$. Also, we consider the dot-action of $W$ on $\h^\ast_\oa$:
\[w\cdot\la = w(\la+\rho_\oa)-\rho_\oa,\text{ for any }w\in W,~\la\in \h_\oa^\ast,\] where $\rho_\oa$ denotes the half-sum of all roots in $\Phi^+_\oa$.   We define the partial order $\leq $ on $\mf h_\oa^\ast$ as the transitive closure of the   relations
$${\la \leq \la+\alpha, ~ \text{~for } \alpha \in \Phi^+.}$$
For $\la\in \mf h_\oa^\ast$, let $\mf u(\la)$ be the (finite-dimensional) irreducible $\mf h$-module of weight $\la$
(see also \cite[Section 1.5.4]{CW12}). We let  $M(\la): =\Ind_{\mf b}^{\mf g} \mf u(\la)$ be the Verma module of highest weight $\la$, where $\mf u(\la)$  is extended to a $\mathfrak{b}$-module by letting $\mathfrak{n}$ act trivially. Let $L(\la)$  be the unique irreducible quotient of $M(\la)$ and denote by $P(\la)$ the projective cover of $L(\la)$ in $\mc O$.

Let $\Pi_\oa$ be the simple system of $\Phi^+_\oa$.  Denote by $\ch \mf n_\oa :=\{\zeta \in \mf n_\oa^\ast|~\zeta([\mf n_\oa, \mf n_\oa])=0 \}$ the set of characters of $\mf n_\oa$.  Associated to a character   $\zeta\in \ch \mf n_\oa$, we define a  Levi subalgebra $\mf l_\zeta$ of $\g_\oa$:  $$ \mf l_\zeta:= \langle \mf h_\oa, \mf g^{\pm\alpha}_\oa|~\alpha \in \Phi^+\text{ such that }\zeta(\g_\oa^\alpha)\neq 0\rangle,$$  with the corresponding simple system $\Pi_\zeta\subseteq \Pi_\oa$.  Let $W_{\mf l_\zeta}$ be the Weyl group  of $\mf l_\zeta$.

  We let $\Lambda \subseteq \h^\ast_\oa$ denote the set of integral weights of $\g_{\bar{0}}$, which coincides with the set of weights appearing in finite-dimensional weight $\g$-modules.   For $\alpha\in \Pi_\oa$, a   module  over $\g$ is said to be $\alpha$-{\em free} if every non-zero vector $f_\alpha\in \g^{-\alpha}$ acts freely on it.

For any Lie superalgebra $\mf s= \mf s_\oa\oplus \mf s_\ob$, we let $U(\mf s)$ denote its universal enveloping algebra. We denote by $\mf s\Mod$ the category of all $\mf s$-modules. For a subalgebra $\mf a\subseteq \mf s$,  we have the usual restriction functor $\Res^{\mf s}_{\mf a}(\_)$ from $\mf s\Mod$ to $\mf a\Mod$.
 It has a left adjoint functor $\Ind_{\mf a}^{\mf s}(\_):=U(\mf s)\otimes\_: \mf a\Mod\rightarrow \mf s\Mod$.
In particular, we define the shorthand notation $\Ind(\_):=\Ind^{\g}_{\mf g_\oa}(\_): \g_\oa\Mod\rightarrow \g\Mod$ and $\Res(\_): = \Res_{\g_\oa}^{\mf g}(\_):  \g\Mod\rightarrow \g_\oa\Mod$. 

  \subsection{Projectively presentable modules}
   Throughout this paper, we fix a dominant
   weight $\nu \in \h^\ast_\oa$ such that  its stabilizer   under the dot-action of $W$ is $W_{\mf l_\zeta}$. Let 
 \[\Lambda(\nu):=\{\la \in \nu+\Lambda|~L(\la) \text{ is $\alpha$-free, }\forall \alpha\in \Pi_\zeta\}.\] We emphasize that a weight in  $\Lambda(\nu)$ is not necessarily an anti-dominant $\mf l_\zeta$-weight. We refer to  \cite[Subsection 4.3.2]{CCC21} for a classification of $\alpha$-free simple highest weight modules over basic classical and strange Lie superalgebras.
 
   The following lemma is proved in the exact same way as \cite[Theorem 4.17]{CCC21}.
 \begin{lem} \label{lem::afree}
 Let $\la  = \la_1\epsilon_1+\la_2\epsilon_2+\cdots +\la_n\epsilon_n\in \nu+\Lambda$, where $\la_1,\ldots, \la_n\in \C$. Then $\la \in \Lambda(\nu)$ if and only if the following conditions hold:
 \begin{enumerate}
     \item[(i)] $(\la,\alpha)\not\in \Z_{>0}$, for any $\alpha\in \Pi_\zeta$; 
      \item[(ii)] If $\la_i=\la_{i+1}$, for some $1\leq i<n$, then $\la_i$= 0.
 \end{enumerate} 
 \end{lem}

 Let $\mc O_{\nu+\Lambda}$ denote the full subcategory of $\mc O$ consisting of all modules in $\mc O$ whose support belongs to $\nu+\Lambda$. A projective module in $\mc O_{\nu+\Lambda}$ is called $\zeta$-admissible if it is a direct sum of projective covers of simple modules of highest  weights in $\Lambda(\nu)$. We define the {\em cokernel} subcategory $\mc O^{\nu\text{-pres}}$ of $\mc O_{\nu+\Lambda}$ to be the full subcategory consisting of all modules $M$  that  admits a two-step presentation of the form
  $P\rightarrow Q\rightarrow M \rightarrow 0,$
  where $P$ and $Q$ are $\zeta$-admissible projective modules in $\mc O_{\nu+\Lambda}$.

  Let $\mc I_\zeta(\nu)$ be the Serre subcategory of $\mc O_{\nu+\Lambda}$ generated by simple objects $L(\la)$ with $\la \in (\nu+\Lambda) \backslash \Lambda(\nu)$. Consider the Serre quotient category $\ov{\mc O}^\nu:=\mc O_{\nu+\Lambda}/\mc I_\zeta(\nu)$ of $\mc O_{\nu+\Lambda}$ by $\mc I_\zeta(\nu)$ in the sense of \cite[Chapter III]{Ga62}, and let  $\pi: \mc O_{\nu+\Lambda}\rightarrow \ov{\mc O}^\nu$ be the associated Serre quotient functor. The following lemma is  a generalization of \cite[Lemma 12]{CCM23} to $\mf q(n)$ with the same proof.

   \begin{lem}
    The functor $\pi$ gives rise to an equivalence from $\mc O^{\nu\text{-pres}}$ to $\ov{\mc O}^\nu$. In particular, $\mc O^{\nu\text{-pres}}$ inherits an abelian category structure, and $\{\pi(L(\la))|~\la\in \Lambda(\nu)\}$ constitutes a complete set of simple objects in $\ov{\mc O}^\nu$. Moreover, $\pi(P(\la))$ is the indecomposable projective cover of $\pi(L(\la))$ in $\ov{\mc O}^\nu$, for any $\la \in \Lambda(\nu)$.
   \end{lem}


 \subsection{BGG reciprocity} \label{sect::22}
   In this subsection, we fix a nil-character $\zeta\in \ch \mf n_\oa$ and establish  a BGG-type reciprocity for the category $\mc O^{\nu\text{-pres}}$. The same results were proved in \cite[Section 5]{CCM23} and \cite[Subsection 4.3]{CC22} under the assumption that $\mf l_\zeta$ is a Levi subalgebra of $\g$.
 \subsubsection{Structural modules in $\mc O^{\nu\text{-pres}}$} \label{sect::232}
We pick an operator $H\in\h_{\oa}$ that determines   the associated  parabolic  decomposition~of~$\g$ in the sense of \cite{CCC21}:  \begin{align}
	&\g=\mf u\oplus \mf l \oplus \mf u^-, \label{para::dec}
\end{align}
such that $\mf l_\oa=\mf l_\zeta$ and $\mf u\subseteq \mf n$, where the subalgebras $\mf u^-$, $\mf l$ and $\mf u$ are given by
 \begin{equation*}\label{deflu}\mf l:=\bigoplus_{\alpha(H)=0} \mf g^\alpha,\quad \mf u:=\bigoplus_{\alpha(H)>0} \mf g^\alpha, \quad \mf u^-:=\bigoplus_{\alpha(H)<0} \mf g^\alpha,\end{equation*}  respectively.  We may note that the Levi subalgebra $\mf l$ is isomorphic to a direct sum of queer Lie superalgebras.
 The   subalgebra $\mf p:= \mf l\oplus\mf u$ is referred to as the  parabolic subalgebra.
 	We emphasize that, in the present paper, we do not assume the condition that $\mf l$ is  purely even, which was imposed in the previous papers \cite{CCM23, CC23_2, CC22}. 

 Consider the BGG category $\mc O(\mf l)$ of $\mf l$-modules with respect to the triangular decomposition $\mf l =(\mf l\cap \mf n)\oplus \mf h \oplus (\mf l \cap \mf n^-)$. Let $L_{\mf l}(\la)$ be the simple object in $\mc O(\mf l)$ of highest weight $\la \in \mf h^\ast_\oa$ and $P_{\mf l}(\la)$ its projective cover. It follows by the Irving-type description of projective-injective modules (see, e.g., \cite[Section 4]{CCC21}),  $P(\la)$ is $\zeta$-admissible if and only if $P_{\mf l}(\la)$ is injective in $\mc O(\mf l)$. 
 
 For any $\mf l$-module $V$,  we define the   parabolically induced $\g$-module  \begin{align}
 	&M^{\mf p}(V):= U(\g)\otimes_{U(\mf p)} V, \label{def::paraind}
 \end{align}  by letting $\mf uV=0.$ The parabolic induction functor $M^{\mf p}(\_): \mf l\Mod\rightarrow \g\Mod$ restricts to a well-defined functor from $\mc O(\mf l) $ to $\mc O$.

 For $\la \in \Lambda(\nu)$, we define the following $\mf g$-modules
 \begin{align}\label{eq::stad}
 	&\Delta(\la): = \Ind_{\mf p}^\g P_{\mf l}(\la),~M^{\mf p}(\la): = \Ind_{\mf p}^\g L_{\mf l}(\la).
 \end{align}

Here, as usual, $P_{\mf l}(\la)$ and $L_{\mf l}(\la)$ are regarded as $\mf p$-modules by letting $\mf u$ act trivially on them. We may note that $\Delta(\la)$ lies in $\mc O^{\nu\text{-pres}}$ by {\cite[Proposition 7]{CC22}}.  The modules $\Delta(\la)$ may be called {\em standard objects} in $\mc O^{\nu\text{-pres}}$ by analogy with the
 standard objects in a properly stratified category of projectively presentable modules in $\mc O$ for other Lie superalgebras; see also \cite[Section 5]{CCM23}. It is worth pointing out that, when $\zeta=0$ (i.e., $W\cdot \nu=\nu$), it was shown that $\mc O^{\nu\text{-pres}} =\mc O_\Lambda$ with standard objects $\Delta(\la)$ and proper standard objects $M^{\mf p}(\la)$ is a properly  stratified category in the sense of \cite{Dl00, MSt04}; see, e.g., \cite{Fr07}.
 \begin{prop} \label{prop::proj} For $\la \in \Lambda(\nu)$, there is a short exact sequence
 	\[0\rightarrow K\rightarrow P(\la) \rightarrow \Delta(\la) \rightarrow 0,\]
 	where $K$ has a filtration, subquotients of which are isomorphic to  $\Delta(\mu)$ with $\mu >\la$.
 \end{prop}
 The proof of Proposition \ref{prop::proj} is postponed until Section \ref{sect::bgg}.

  \subsubsection{BGG reciprocity for $\mc O^{\nu\text{-pres}}$}  \label{sect::bgg} For $\la \in \h_\oa^\ast$, we define \begin{align}
  	&N(\la):=\Ind_{\mf b}^{\mf g} \hat{\mf u}(\la),\label{eq::defofN}
  \end{align} where $\hat{\mf u}(\la)$ denotes  the projective cover of $\mf u(\la)$
  in the category of $\mf h$-supermodules that are semisimple over $\mf h_\oa$.

  For any $\g$-module $M$ having a composition series and any simple $\g$-module $L$, we let $[M : L]$  denote the composition multiplicity of $L$ in $M$. If $M$ admits a filtration of modules such that subsequent quotients are of the form $N(\la)$, $\la\in\h^\ast_\oa$, we define $(M: N(\la))$ to be the multiplicity of $N(\la)$ appearing in that filtration. Similarly, we define the multiplicity $(M: \Delta(\la))$ for a module $M$ that has a filtration with subsequent quotients of the form $\Delta(\la)$, $\la\in\h^\ast_\oa$. The following lemma is the BGG reciprocity for category $\mc O$ of $\mf q(n)$; see, e.g,  \cite[Example 7.10]{Bru04b} or \cite[Corollary 3]{Maz14}. 
 \begin{lem}[Brundan] \label{lem::BGGforq}
 	For $\la, \mu \in \Lambda$, we have
 	\begin{align*}
 		&(P(\la): N(\mu)) =[M(\mu): L(\la)].
 	\end{align*}
 \end{lem}
Note that in the special case where  $\mf l =\mf h$. $\Delta(\la) = N(\la)$ and $M^{\mf p}(\la) =M(\la)$. The following BGG-type reciprocity for $\mc O^{\nu\text{-pres}}$ extends Lemma \ref{lem::BGGforq} and generalizes results in \cite[Proposition 27]{CCM23} and \cite[Corollary 16]{CC22}, where  $\mf l$ is assumed to be purely even.
Proposition \ref{prop::proj} follows as a direct consequence of this result.
\begin{prop}[BGG reciprocity] \label{prop::BGGforpres}  For any $\la \in \Lambda(\nu)$, the projective cover $P(\la)$ has a $\Delta$-flag. Furthermore,
	the following BGG-type reciprocity holds
	\begin{align}
		&(P(\la): \Delta(\mu)) =[M^{\mf p}(\mu): L(\la)], \hskip0.2cm \text{ for $\la, \mu \in \Lambda(\nu)$.} \label{eq::BGGforpres}
	\end{align}
\end{prop}
\begin{proof} First, we shall show that any $\zeta$-admissible projective cover $P(\la)$ in $\mc O_{\nu+\Lambda}$ has a $\Delta$-flag.
  By the same argument for \cite[Proposition 13]{CC22}, it follows that  the full subcategory of $\mc O_{\nu+\Lambda}$ consisting of modules which admit a $\Delta$-flag is closed under taking direct summand. Therefore, to prove the first assertion, it suffices to show that $\Ind \Res P(\la)$ has a $\Delta$-flag.

  For $\la\in \Lambda(\nu)$, let $\Delta^{0}(\la)$  the  $\g_\oa$ analogue of the standard objects $\Delta$ in \eqref{eq::stad}, that is,  $\Delta^0(\la) = \Ind_{\mf p_\oa}^{\g_\oa}P_{\mf l_\zeta}(\la)$, where $P_{\mf l_\zeta}(\la)$ denotes the projective cover of the simple highest weight $\mf l_\zeta$-module with highest weight $\la$.  Since  $\Res P(\la)$ is a direct sum of projective covers of $\alpha$-free irreducible $\g_\oa$-modules, for all $\alpha\in\Pi_\zeta$, we see that it has a  $\Delta^0$-flag. It then suffices to show that $\Ind \Delta^0$ has a $\Delta$-flag. To prove this, we identify $P_{\mf l_\zeta}(\la)$ with the $\mf l_\zeta$-submodule $1_{U(\g)}\otimes1_{U(\g_\oa)}\otimes  P_{\mf l_\zeta}(\la)\subset 1_{U(\g)} \otimes \Delta^0(\la)\subset \Ind \Delta^0(\la)$. Set  $F: = U(\mf u_\ob)$, which is an $\mf l_\zeta$-submodule of $U(\g)$ under the adjoint action. Then the $\mf l_\zeta$-submodule  $U(\mf u_\ob)P_{\mf l_\zeta}(\la)$   of $\Ind \Delta^0(\la)$ is isomorphic to $F\otimes P_{\mf l_\zeta}(\la)$, which is a projective-injective object in the category $\mc O(\mf l_\zeta)$ of $\mf l_\zeta$-modules. Therefore, the $\mf p$-submodule $N:= U(\mf l_\ob) U(\mf u_\ob)P_{\mf l_\zeta}(\la)\cong \Ind_{\mf l_\zeta}^{\mf l}(F\otimes P_{\mf l_\zeta}(\la))$ of $\Ind \Delta^0(\la)$ restricts to  a $\mf l$-module which is  a projective-injective module in $\mc O(\mf l)$.   Recall the grading operator $H\in \h_\oa$ from Subsection \ref{sect::232}.	Since $H$ lies in the center of $\mf l$, it follows that $H$ acts on $P_{\mf l}(\la)$ as a scalar $c$. Therefore, the $\mf l$-module $N$ constructed above decomposes into a direct sum $N=\bigoplus_{d\in \C,~d<c} N_{d}$ of $H$-eigenspaces, where $N_d$ is the eigenspace correspondes to the eigenvalue $d$. We may conclude that  $N$ admits a filtration of $\mf p$-submodules
  \[0 = N(0)\subset  N(1) \subset \cdots \subset N(\ell) =N,\]
  such that each $N^k: = N(k+1)/ N(k)$ restricts to a projective-injective object in $\mc O(\mf l)$ and $\mf u N^k=0$. Consequently, we obtain a filtration of $\g$-modules $$0=U(\mf u^-)N(0)\subset U(\mf u^-)N(1)\subset \cdots \subset U(\mf u^-)N(\ell) =\Ind\Delta^0(\la),$$
  in which each subquotient is isomorphic to  standard object $\Delta(\eta)$, for $\eta\in \Lambda(\nu)$, by considering weights of $\Ind \Delta^0(\la)$. This proves the first assertion.

 Next, we shall prove the BGG reciprocity in \eqref{eq::BGGforpres} as follows. 	We  let $[\mc O]$ denote the Grothendieck group of $\mc O$. For an object $M\in \mc O$,  let us denote by $[M]$ the corresponding element in $[\mc O]$.	Define the following integers $$\alpha_{\la\mu}: = (P(\la):\Delta(\mu)),~\beta_{\la\mu} = (M^{\mf p}(\mu): L(\la)),~ \text{and}~\gamma_{\mu\eta} := [M_{\mf l}(\eta):L_{\mf l}(\mu)],$$  for $\la, \mu,\eta\in \Lambda$, and form the matrices $(\alpha_{\la\mu})_{\la,\mu\in \Lambda(\nu)}, (\beta_{\la\mu})_{\la,\mu\in \Lambda(\nu)}$ and $ (\gamma_{\mu\eta})_{\mu,\eta\in \Lambda(\nu)}$.
	
	 We  define $N_{\mf l}(\eta)$ to be the $\mc O(\mf l)$ analogue of the module $N(\eta)$ in \eqref{eq::defofN} for $\g$, for any $\eta \in \Lambda(\nu)$. We calculate
	\begin{align*}
	&[P(\la)] =\sum_{\mu} (P(\la):\Delta(\mu)) [\Delta(\mu)] =\sum_{\mu, \eta} (P(\la):\Delta(\mu)) (P_{\mf l}(\mu): N_{\mf l}(\eta)) [N(\eta)].
	\end{align*} This, together with
 an analogue of BGG reciprocity of Lemma \ref{lem::BGGforq} for $\mc O(\mf l)$, implies  that  \begin{align}
&(P(\la): N(\eta)) = \sum_{\mu}(P(\la): \Delta(\mu))[M_{\mf l}(\eta): L_{\mf l}(\mu)] =\sum_\mu \alpha_{\la\mu}\gamma_{\mu\eta}. \label{eq::2}
\end{align} Next,  expanding $[M(\eta)]$ in terms of $[M^{\mf p}(\kappa)]$ yields  
$[M(\eta)] = \sum_\kappa [M_{\mf l}(\eta): L_{\mf l}(\kappa)] [M^{\mf p}(\kappa)]$ and 
\begin{align}
	&[M(\mu): L(\la)] = \sum_\kappa [M_{\mf l}(\eta): L_{\mf l}(\kappa)] [M^{\mf p}(\kappa): L(\la)] =\sum_{\mu}\beta_{\la\mu}\gamma_{\mu\eta}.
\end{align} By Lemma \ref{lem::BGGforq}, we have
\[\sum_\mu \alpha_{\la\mu}\gamma_{\mu\eta}= \sum_{\mu}\beta_{\la\mu}\gamma_{\mu\eta},\] for $\la \in \Lambda(\nu)$. This implies that  $$((\alpha_{\la\mu})_{\la,\mu\in \Lambda(\nu)} -  (\beta_{\la\mu})_{\la,\mu\in \Lambda(\nu)})  (\gamma_{\mu\eta})_{\mu,\eta\in \Lambda(\nu)} =0.$$ Consequently, we have $(\alpha_{\la\mu})_{\la,\mu\in \Lambda(\nu)}  = (\beta_{\la\mu})_{\la,\mu\in \Lambda(\nu)}$ and the conclusion follows.
\end{proof}

 \begin{rem}
 	The category $\mc O^{\nu\text{-pres}}$ equipped with the standard objects $\Delta(\la)$ and the proper standard object $M^{\mf p}(\la)$ is usually not properly stratified in general. For example, consider  $\g =  \mf q(2)$  with $\nu=0$. In this case, we have $\Delta(\la) =P(\la)$ and  $M^{\mf p}(\la) =L(\la)$, for any $\la\in \h^\ast_\oa$. Now, let $a<-1$ be an integer and consider $\la = a\varepsilon_1- a\varepsilon_2$, which lies in $\Lambda(\nu)$ by  Lemma \ref{lem::afree}. Then, it follows from \cite[Lemma~2.44]{CW12} and the BGG reciprocity that $\Delta(\la) = P(\la)$ has a composition factor isomorphic to $L((\la -(\epsilon_1-\epsilon_2))\not\cong L(\la)$. However, if $\mc O^{\nu\text{-pres}}$ were properly stratified, then the standard object $\Delta(\la)$ would admit a filtration subquotients of which are isomorphic to the proper standard object $M^{\mf p}(\la)$.
 \end{rem}

\section{Whittaker modules for queer Lie superalgebras} \label{sect::2}
We keep the notation and assumptions of the previous sections. In this section, we first introduce the MMS Whittaker category $\mc N(\zeta)$ of $\mf q(n)$-modules, and then establish a bijection between simple objects in $\mc N(\zeta)$ and those in the MMS Whittaker category $\mc N_{\mf l}(\zeta)$ of $\mf l$-modules. Then we define standard Whittaker modules for $\mf q(n)$ and study Backelin functor $\Gamma_\zeta: \mc O\rightarrow \mc N(\zeta)$. As a consequence, we obtain a reduction of the multiplicity problem for standard Whittaker modules to that of Verma modules.

\subsection{McDowell-Mili{\v{c}}i{\'c}-Soergel type  Whittaker category $\mc N(\zeta)$}
Following the work of McDowell \cite{Mc85} and Mili{\v{c}}i{\'c}--Soergel \cite{MS97} on reductive Lie algebras, we define the Whittaker category $\mc N$ associated with the triangular decomposition $\g =\mf n\oplus \mf h\oplus \mf n^-$. This category consists of finitely-generated $\g$-modules on which the actions of $U(\mf n)$ and $Z(\g_\oa)$ are locally finite.
For any  nil-character $\zeta\in \ch \mf n_\oa$, let $\mc N(\zeta)$ denote the full subcategory of $\mc N$ consisting of modules   $M\in \mc N$ such that  $x-\zeta(x)$ acts locally nilpotently on $M$, for each $x\in \mf n_\oa$. We have the decomposition
\[\mc N=\bigoplus_{\zeta\in \ch \mf n_\oa}\mc N(\zeta).\]
By definition, $\mc N(\zeta)$ depends only on the even~part~$\mf n_\oa$~of~$\mf n$.

Recall the Levi subalgebra $\mf l$ from \eqref{para::dec} and its triangular decomposition:
\begin{align*}
&\mf l =(\mf l\cap \mf n) \oplus \h \oplus (\mf l\cap \mf n^-).
\end{align*} We let $\mc N_{\mf l}(\zeta)=\mc N_{\mf l}(\zeta|_{\mf l_\zeta\cap \mf n})$ denote the corresponding MMS Whittaker category of $\mf l$-modules associated to the character $\zeta|_{\mf l_\zeta\cap \mf n}$.  The categories $\mc N_{\g_\oa}(\zeta)$ and  $\mc N_{\mf l_\oa}(\zeta)$ are defined analogously.
The induction and restriction functors $(\Ind(\_), \Res(\_))$form an adjoint pair between     $\mc N_{\mf g_\oa}(\zeta)$ and $\mc N(\zeta)$.

\subsection{Classification of simple objects in $\mc N(\zeta)$} \label{sect::clsofsimpleWhi} For any abelian category $\mc C$,   let $\rm{Irr}(\mc C)$ denote the set of isomorphism classes of simple objects in $\mc C$.

  The following (simple) Whittaker module was introduced by Kostant in \cite{Ko78}:
\begin{align}
	&Y_\zeta(\la, \zeta):=U(\mf l_\zeta)/\text{Ker}(\chi^{\mf l_\zeta}_\la) U(\mf l_\zeta) \otimes_{U(\mf n_\oa\cap \mf l_\zeta)}\mathbb C_\zeta, {\text{ for }\la \in \h_\oa^\ast},
\end{align} where  $\text{Ker}(\chi^{\mf l_\zeta}_\la)$ is the kernel of the central character $\chi_\la^{\mf l_\zeta}$ of $\mf l_\zeta$ and $\C_\zeta$ is the one-dimensional $\mf n_\oa\cap \mf l_\zeta$-module associated with $\zeta$. We define
$M_0(\la,\zeta)$ to be the classical standard Whittaker module  in the sense of \cite{Mc85} and  \cite[Section 1]{MS97}:
\[M_0(\la,\zeta) = U(\g_\oa)\otimes_{U(\mf p_\oa)}Y_\zeta(\la,\zeta),\]
where   $\mf u_\oa$ acts trivially on $Y_\zeta(\la, \zeta)$.

For any simple Whittaker module $V\in \mc N_{\mf l}(\zeta)$,     the parabolically induced modules $M^{\mf p}(V)$ will be referred to as {\em standard Whittaker modules} for $\mf q(n)$   by analogy with the standard Whittaker modules of Lie algebras and other Lie superalgebras (see also \cite[Section~3.1]{C21}).

\begin{lem} \label{lem::simpletop9} If $V\in \mc N_{\mf l}(\zeta)$, then    $M^{\mf p}(V)$ is an object in $\mc N(\zeta)$.  Furthermore, if $V\in \rm{Irr}(\mc N_{\mf l}(\zeta))$ then $M^{\mf p}(V)$ has a unique maximal submodule $N(V)$.
\end{lem}
\begin{proof}
	By assumption, we have $\Res_{\mf l_\oa}^{\mf l}V\in \mc N_{\mf l_\oa}(\zeta)$. Since 
    $$\Hom_{\g}(\Ind_{\mf p_\oa}^\g \Res^{\mf l}_{\mf l_\oa}V,\Ind_{\mf p}^\g V)\cong\Hom_{\mf p_\oa}(\Res^{\mf l}_{\mf l_\oa}V, \Res^{\g}_{\mf p_\oa}\Ind_{\mf p}^\g V)\not=0,$$
    we see that $M^{\mf p}(V)$ is a quotient of $\Ind^{\g}_{\g_\oa}(U(\g_\oa)\otimes_{\mf p_\oa}\Res_{\mf l_\oa}^{\mf l}V)$, where  $\mf u_\oa$ acts on $\Res_{\mf l_\oa}^{\mf l}V$ trivially. Since \mbox{$U(\g_\oa)\otimes_{\mf p_\oa}\Res_{\mf l_\oa}^{\mf l}V$} has a filtration of the classical standard Whittaker modules over $\g_\oa$ in the sense of \cite{Mc85,MS97}, it is an object in $\mc N_{\mf g_\oa}(\zeta)$.  This proves that $M^{\mf p}(V)\in \mc N(\zeta)$.
	
	To show the uniqueness of the maximal submodule of $M^{\mf p}(V)$, for $V\in \rm{Irr}(\mc N_{\mf l}(\zeta))$, we use the grading operator   $H\in \h_\oa$ from Subsection \ref{sect::232}.  Since  $H\in Z(\mf l)$, it acts on $V$ by a scalar $c = c_{H,V}\in \C$. We   decompose $M^{\mf p}(V)$ into $H$-eigenspaces: $$M^{\mf p}(V) =\bigoplus_{k{\leq} 0}M^{\mf p}(V)_{c+k},$$ where each $M^{\mf p}(V)_z = \{x\in M^{\mf p}(V)|~Hx=zx\}$ is the $z$-eigenspace of $M^{\mf p}(V)$, with respect to the action of $H$ on $M^{\mf p}(V)$. Any submodule $N$  of $M^{\mf p}(V)$ must be graded by these eigenspaces, namely, $N = \bigoplus_{z}N_z$, where $N_z =M^{\mf p}(V)_z\cap N$. Since   the top component $M^{\mf p}(V)_c=V$ is a simple $\mf l$-module, it follows that $N=M^{\mf p}(V)$ if and only if $N_c\neq 0$. This implies  that $M^{\mf p}(V)$ has a unique maximal submodule.   
\end{proof}

\begin{rem} \label{rem::8}
	For any $V\in \rm{Irr}(\mc N_{\mf l}(\zeta))$, we define $L(V):=M^{\mf p}(V)/N(V)$, where $N(V)$ is the unique maximal submodule of $M^{\mf p}(V)$. Since $L(V)=\bigoplus_{k{\leq}0}L(V)_{c+k}$ and $L(V)_{c}=V$, it follows that \[L(V)\cong L(W) \Leftrightarrow V\cong W, \]
	for $V, W\in \rm{Irr} \mc N_{\mf l}(\zeta)$.
\end{rem}

We note that \begin{align}
	&\mathfrak{l}\cong \mathfrak{q}(n_1)\oplus \mathfrak{q}(n_2)\oplus \cdots \oplus \mathfrak{q}(n_{k})  \label{eq::Levi}
\end{align} for some positive integers $n_1, n_2\ldots,n_k$ with $\sum_{i=1}^k n_i = n$ such that $\zeta|_{\mf n_\oa\cap \mf q(n_i)}$ is non-singular, for any $1\leq i\leq k.$  Below we shall give a reduction of the description of irreducible Whittaker modules in $\mc N_{\mf l}(\zeta)$ to that of the categories  $\mc N_{\mf q(n_i)}(\zeta),$ for $1\leq i\leq k$. To give the description, we let $\mf a$ and $\mf b$ be two summands of $\mf l$ in the expression \eqref{eq::Levi}  such that $\mf l =\mf a\oplus \mf b$. Then $\mf a\cap \mf n$ and $\mf b\cap \mf n$ form nilradicals of Borel subalgebras of $\mf a$ and $\mf b$,  respectively. We define the Whittaker categories $\mc N_{\mf a}(\zeta)$ and $\mc N_{\mf b}(\zeta)$ in a similar fashion.

The following is a Whittaker analogue of \cite[Proposition 2.10]{Joz88} and \cite[Proposition 8.4]{C95}.  Let $\mf s= \mf a, \mf b$. A given $\mf s$-module $L$ is called type \texttt{M} if $\dim \End_{\mf s}(L)=1$, and it is called type \texttt{Q} if $\dim \End_{\mf s}(L)=2$.
\begin{prop} \label{prop::LeviSimple}
	For any simple Whittaker modules $V_{\mf a}\in \mc N_{\mf a}(\zeta)$ and $V_{\mf b}\in\mc N_{\mf b}(\zeta)$, the $\mf l$-module  $V_{\mf a}\otimes V_{\mf b}$ is either simple or a direct sum of two isomorphic simple modules over $\mf l$. Furthermore, if we define $V_{\mf a}\hat \otimes V_{\mf b}\subseteq V_{\mf a}\otimes V_{\mf b}$ to be a non-zero simple submodule, then these simple modules constitute a complete set of mutually non-isomorphic simple objects in $\mc N_{\mf l}(\zeta)$.
\end{prop}
\begin{proof}
		For any given simple Whittaker modules $V_{\mf a}\in \mc N_{\mf a}(\zeta)$ and $V_{\mf b}\in\mc N_{\mf b}(\zeta)$, we note  that $V_{\mf a}\otimes V_{\mf b}$ lies in $\mc N_{\mf l}(\zeta)$. Then, $V_{\mf a}\otimes V_{\mf b}$ contains a simple Whittaker submodule $V$. By an argument similar to that used in \cite[Proposition 8.4]{C95} (see also the Remark below \cite[Proposition 8.4]{C95}), it follows that $V_{\mf a}\otimes V_{\mf b}$ is simple if one of $V_{\mf a}$ and $V_{\mf b}$ is of type \texttt{M}, and $V_{\mf a}\otimes V_{\mf b}\cong V\oplus V$, for some simple $\mf l$-module $V$, if both $V_{\mf a}$ and $V_{\mf b}$ are of type \texttt{Q}.

	Let $V\in \mc N_{\mf l}(\zeta)$ be a simple object. Since $\Res^{\mf l}_{\mf l_\zeta} V \in \mc N_{\mf l_\zeta}(\zeta)$, the $\mf l_\oa$-Whittaker vector subspace $\text{Wh}_\zeta^0(V):=\{v\in V|~xv=\zeta(x)v,~\text{for any }x\in \mf l\cap \mf n_\oa\}$ is finite-dimensional, and therefore $U(\mf n)\text{Wh}_\zeta^0( V)$ is a finite-dimensional $\mf n$-submodule of $V$. By \cite[Lemma 1.37]{CW12}, there exists a non-zero $\mf l$-Whittaker vector  $v\in V$. 
	
	Consider the $\mf a$-module $U(\mf a)v$. Since $Z(\mf a_\oa)v$ is finite-dimensional, it follows that $U(\mf a)v$, which is an epimorphic image of $\Ind_{\mf a_\oa}^{\mf a}U(\mf a_\oa)v$, is locally finite over $Z(\mf a_\oa)$. Therefore, $U(\mf a)v\in \mc N_{\mf a}(\zeta)$ is of finite-length and contains a simple $\mf a$-submodule $V_{\mf a}$. Similarly, $V$ contains a simple $\mf b$-submodule $V_{\mf b}$.
	
	By the PBW basis theorem, $V = U(\mf b)V_{\mf a} =\bigoplus r_iV_{\mf a}$, where $r_i$ runs over some of the PBW basis elements in $U(\mf b)$. Hence, $V$ is a direct sum of simple $\mf a$-submodules $r_iV$ isomorphic to $V_{\mf a}$, up to parity-change. Similarly, $V$ is a direct sum of simple $\mf b$-submodules isomorphic to $V_{\mf b}$. The space $\Hom_{\mf a}(V_{\mf a}, V)$ has  a natural $\mf b$-module structure.  Consider the $\mf l$-module  $\Hom_{\mf a}(V_{\mf a}, V)\otimes V_{\mf a}$ with the even epimorphism between $\mf l$-modules determined by
	\begin{align*}
	&\phi: \Hom_{\mf a}(V_{\mf a}, V)\otimes V_{\mf a}\longrightarrow V,~\phi: f\otimes v\mapsto f(v).
	\end{align*}  If $V_{\mf a}$ is of type \texttt{M}, then each $\Hom_{\mf a}(V_{\mf a}, r_iV_{\mf a})$ is spanned by the map $v\mapsto r_{i}v$, for $v\in V$. Therefore $\phi$ is an isomorphism. In this case, we note that  $\Hom_{\mf a}(V_{\mf a}, V)$ is a simple $\mf b$-module, and so it is isomorphic to $V_{\mf b}.$ Similar conclusion holds for the case that $V_{\mf b}$ is of type \texttt{M}. In these cases, we may conclude that $V\cong V_{\mf a}\otimes V_{\mf b}$.

 Now we suppose that both $V_{\mf a}$ and $V_{\mf b}$ are of type Q. In this case, we claim that $\Hom_{\mf a}(V_{\mf a}, V)$ is a direct sum of simple $\mf b$-submodule. To see this, let $\sigma$ be an odd automorphism of $ V_{\mf a}$ such that $\sigma ^2=1$. We note that $f(-)\mapsto f\circ \sigma(-)$, for $f\in  \Hom_{\mf a}(V_{\mf a}, V)$, defines a $(U(\mf b), \C[\sigma])$-bimodule structure on $\Hom_{\mf a}(V_{\mf a}, V)$. Since $V$ is a simple $\mf l$-module, it follows that each homomorphism in $\Hom_{\mf a}(V_{\mf a}, V)$ is spanned by the maps $v\rightarrow r_i v$ or $v\rightarrow r_i\sigma(v)$. Hence, $\Hom_{\mf a}(V_{\mf a}, V)$ is simple as a $(U(\mf b), \C[\sigma])$-module. In particular, this implies that $\Hom_{\mf a}(V_{\mf a}, V)$ is a direct sum of simple $\mf b$-modules. Consequently, $V$ is a homomorphic image of a tensor product of simple $\mf a$- and $\mf b$-modules. This proves that $V\cong V_{\mf a}\hat \otimes V_{\mf b}$, up to parity-change.

 Finally, suppose that $V_{\mf a}\hat\otimes V_{\mf b} \cong V'_{\mf a}\hat\otimes V'_{\mf b}$. Since $V'_{\mf a}\otimes V'_{\mf b}$ is a direct sum of simple $\mf a$-submodules isomorphic to $V'_{\mf a}$, it follows that $V_{\mf a}\hat\otimes V_{\mf b}$ contains a simple $\mf a$-submodule isomorphic to $V'_{\mf a}$. It follows that $V_{\mf a}\cong V'_{\mf a}$. Similarly, we have $V_{\mf b}\cong V'_{\mf b}$. This completes the proof. 
\end{proof}

In Remark \ref{rem::clsPS16} below we shall explain an approach to describe simple objects in $\mc N_{\mf l}(\zeta)$ based on the classification of irreducible representations of the principal finite $W$-algebras for $\mf q(n)$ given in \cite{PS16b}.

The following is the main result in this section:
\begin{thm} \label{cls::thm}
	The correspondence $$V\mapsto L(V),~\text{ for }~V\in \rm{Irr}\mc N_{\mf l}(\zeta)$$  gives rise to a bijection between $\rm{Irr}\mc N_{\mf l}(\zeta)$ and $\rm{Irr}\mc N(\zeta)$.
\end{thm}
\begin{proof}   Let $S$ be a simple module in $\mc N(\zeta)$. By Remark \ref{rem::8}, it remains to show that $S$ is isomorphic to $L(V)$, for some $V\in \rm{Irr}\mc N_{\mf l}(\zeta)$.  By definition, we have $\Res S \in \mc N_{\g_\oa}(\zeta)$. Therefore, it follows by adjunction between $\Ind$ and $\Res$ that $S$ is a quotient of $\Ind M_0(\la,\zeta)$, for some $\la\in \h^\ast_\oa$.
	Since $\Res \Ind M_0(\la,\zeta)\cong \Lambda(\g_\ob)\otimes M_0(\la,\zeta)$, there exists $r\in \mathbb C$ such that  $\Ind M_0(\la,\zeta)$ decomposes into a direct sum  of eigenspaces with respect to the action of $H$:  $$\Ind M_0(\la,\zeta)= \bigoplus_{k\leq 0}\Ind M_0(\la,\zeta)_{r+k}$$  where
	$$\Ind M_0(\la,\zeta)_z = \{x\in \Ind M_0(\la,\zeta)|~Hx =zx\}, \text{ for } z\in \C.$$ Note that each $H$-eigenspace $\Ind M_0(\la,\zeta)_z$   is an $\mf l$-submodule of $\Res_{\mf l}^{\mf g}\Ind M_0(\la,\zeta)$. Since $\mc N_{\mf l_\oa}(\zeta)$ is closed under tensoring with finite-dimensional modules, it follows that  $\Res_{\mf l_\oa}^{\mf l}\Ind M_0(\la,\zeta)_z$ has  finite length (see, e.g., \cite[Theorem~4.6]{Ko78}). This gives rise to a decomposition  of $\Res^{\g}_{\mf l}S$: \[\Res^{\g}_{\mf l}S =\bigoplus_{k{\leq}0}S_{r'+k},~\text{where}~S_{r'+k} = S\cap \Ind M_0(\la,\zeta)_{r'+k},\] where $r'\in \C$ is the maximum eigenvalue such that $S_{r'}\neq 0$. Then $S_{r'}\in \mc N_{\mf l}(\zeta)$ satisfying $\mf u S_{r'} =0$.   By choosing a simple $\mf l$-submodule $V$ of $S_{r'}$, we have  $V\in \rm{\mc N_{\mf l}(\zeta)}$ and  $\mf uV=0$. By adjunction, $S$ is a simple quotient of $M^{\mf p}(V)$ and so $S\cong L(V)$. This completes the proof.
\end{proof}

\subsection{Backelin functor and standard Whittaker modules} \label{sect::bac} Throughout this subsection, we fix a character $\zeta\in \ch \mf n_\oa$ and let  $\mf l_\zeta$ be the subalgebra defined in Subsection~\ref{sect::22}. We keep the notation and assumptions of the previous subsections.

  For any $M\in \mc O$, we let $\ov M:= \prod_{\la\in \h^\ast_\oa} M_\la$ denote the completion of $M$ with respect to its weight spaces $M_\la$.
We recall the super analogue of Backelin’s functor $\Gamma_\zeta: \mc O\rightarrow \mc N(\zeta)$ from \cite[Section 5.2]{C21}, which is a natural extension of Backelin’s original functor from \cite[Section 3]{B97} for reductive Lie algebras:
\[\Gamma_\zeta: \mc O\rightarrow \mc N(\zeta),~M\mapsto \{m\in \ov M|~x-\zeta(x)\text{ acts nilpotently on $m$, for any } x\in \mf n_\oa \}.\]
The functor $\Gamma_\zeta$  has been
shown to be  a useful tool for the study of the composition multiplicity of standard Whittaker modules in the setting of a quasi-reductive Lie superalgebra under the additional assumptions that $\mf l_\zeta$ is a Levi subalgebra of $\g$ (i.e., $\mf l_\zeta =\mf l_\oa=\mf l$); see, e.g., \cite[Theorem 1]{CC22}. The goal of this subsection is to develop an analogue for the queer Lie superalgebra $\mf q(n)$.

We denote by $\Gamma_\zeta^{\mf g_\oa}: \mc O(\mf g_\oa)\rightarrow \mc N_{\mf g_\oa}(\zeta)$ and  $\Gamma_\zeta^{\mf l}: \mc O(\mf l)\rightarrow \mc N_{\mf l}(\zeta)$ the Backelin functor for  $\mf g_\oa$-modules and for $\mf l$-modules, respectively. The following proposition shows that the Backelin functor commutes with the parabolic induction functor, extending the results in \cite[Theorem 20]{C21} and \cite[Proposition 4]{CC22} to $\mf q(n)$.
\begin{prop} \label{prop::bacpara}
Suppose that $V\in \mc O(\mf l)$ is $\alpha$-free, for any simple root $\alpha$ of $\mf l_\zeta$. Then we have
\begin{align}
&\Gamma_\zeta(M^{\mf p}(V)) \cong M^{\mf p}(\Gamma_\zeta^{\mf l}(V)).
\end{align}
\end{prop}
\begin{proof}
	In the proof, we may identify $\Gamma_\zeta^{\mf l}(V)$ as the $\mf l$-submodule $1\otimes \Gamma_\zeta^{\mf l}(V)$ of $M^{\mf p}(\Gamma_\zeta^{\mf l}(V))$.  Then there is a  natural inclusion of $\g$-modules $M^{\mf p}(\Gamma_\zeta^{\mf l}(V))  \cong U(\mf u^-)\Gamma_\zeta^{\mf l}(V) \hookrightarrow \Gamma_\zeta(M^{\mf p}(V))$. We claim that this is indeed an isomorphism.
	
	We may assume that $V$ is indecomposable.  Recall the grading operator $H\in \h_\oa$ from Subsection \ref{sect::232}. For any  $X\in \mc N(\zeta)$, let  $X=\oplus_{c\in \C} X_c$ be the decomposition of $X$ into generalized $H$-eigenspaces, where $X_c$ corresponds to the eigenvalue $c$ under action of $H$. Since each $X_c$ is an $\mf l$-module lying in $\mc N_{\mf l}(\zeta)$,  it suffices to show that $\Gamma_\zeta(M^{\mf p}(V))_c$ and $M^{\mf p}(\Gamma_\zeta^{\mf l}(V))_c$ have the same composition factors for all $c \in \mathbb{C}$.
	
	We  let $[\mc N_{\mf s}(\zeta)]$ denote the Grothendieck group of $\mc N_{\mf s}(\zeta)$, for $\mf s= \g_\oa, \mf l_\zeta$. For $X\in \mc N_{\mf s}(\zeta)$, the expression $[X]$ denotes the corresponding element in $[\mc N_{\mf s}(\zeta)]$. By assumption, $\Res_{\mf l_\zeta}^{\mf l}V$ admits a $\mf l_\zeta$-Verma flag with consecutive subquotients isomorphic to  \[M_{\mf l_\zeta}(\la_1),\ldots, M_{\mf l_\zeta}(\la_k),\] where   $M_{\mf l_\zeta}(\la_i)$ denotes the $\mf l_\zeta$-Verma module of highest weight $\la_i$. Furthermore, we set $E: = U(\mf u_\ob^{-})$, viewed as an $\mf l_\zeta$-submodule of $U(\g)$ via the adjoint action, and denote its set of weights by $\Phi(E)$. We have the following identities  in $[\mc N_{\g_\oa}(\zeta)]$:
	$$\begin{array}{rcl}
		[\Res \Gamma_\zeta M^{\mf p}(V)] & \overset{}{=}& [\Gamma_\zeta^{\g_\oa} \Res M^{\mf p}(V)] \\ &= &\sum_{i=1}^k [\Gamma_\zeta^{\g_\oa} \Ind_{\mf p_\oa}^{\mf g_\oa}(E\otimes M_{\mf l_\zeta}(\la_i))] \\
		&\overset{\text{by $\dim E< \infty$}}{=}&\sum_{i=1}^k \sum_{\gamma\in \Phi(E)}[\Gamma_{\zeta}^{\g_\oa} \Ind_{\mf p_\oa}^{\mf g_\oa} M_{\mf l_\zeta}(\la_i+\gamma)]\\
		&\overset{\text{by \cite[Proposition 6.9]{B97}}}{=}&\sum_{i=1}^k\sum_{\gamma \in \Phi(E)}[M_0(\la_i +\gamma,\zeta)]. 
	\end{array}$$

On the other hand, we observe that $\Res^\g_{\mf l_\zeta} M^{\mf p}(\Gamma_\zeta^{\mf l}(V))\cong U(\mf u_\oa^-)\otimes E\otimes \Gamma_\zeta^{\mf l}(V).$ By  \cite[Proposition 6.9]{B97} and \cite[Lemma 5.12]{MS97}, it follows that   $$[E\otimes \Gamma_\zeta^{\mf l}(V)] =  \sum_{i=1}^k [E\otimes M_{\mf l_\zeta}(\la_i, \zeta)] = \sum_{i=1}^k  \sum_{\gamma\in \Phi(E)} [M_{\mf l_\zeta}(\la_i+\gamma, \zeta)] \hskip0.1cm \text{ in $[\mc N_{\mf l_\zeta}(\zeta)]$.}$$

For any $c\in \C$ and $A\in \mc N_{\mf l_\zeta}(\zeta)$,  let $[A]_c$ denote the element in $[\mc N_{\mf l_\zeta}(\zeta)]$ corresponding to  the $\mf l_\zeta$-module $A_c$. The above calculations imply the following equalities of elements in $[\mc N_{\mf l_\zeta}(\zeta)]$:
$$\begin{array}{rcl}
 [\Res_{\mf l_\zeta}^{\g} M^{\mf p}(\Gamma_\zeta^{\mf l} (V))]_c & \overset{}{=} &  \sum_{i=1}^k \sum_{\gamma\in \Phi(E)} [U(\mf u_\oa^-)\otimes M_{\mf l_\zeta}(\la_i +\gamma, \zeta)]_c \\  &=& \sum_{i=1}^k \sum_{\gamma\in \Phi(E)}[M_0(\la_i+\gamma,\zeta)]_c \\
 & = & [\Res_{\mf l_\zeta}^\g \Gamma_\zeta(M^{\mf p}(V))]_c.
	\end{array}$$
This completes the proof.
\end{proof}

Recall that $\nu \in \h^\ast_\oa$ is dominant weight such that its stabilizer subgroup is $W_{\mf l_\zeta}$ under the dot-action of $W$. For any $\la\in \Lambda(\nu)$, the Whittaker $\mf l$-module $\Gamma_\zeta^{\mathfrak{l}}(L_{\mathfrak{l}}(\lambda))$ is known to be simple by \cite[Theorem 33]{CC23_2}.  We provide some further details regarding this irreducible module. According to the decomposition \eqref{eq::Levi} of $\mf l$ into a direct sum of ideals $\bigoplus_{i=1}^k \mathfrak{q}(n_i)$, the simple highest weight module  $L_{\mf l}(\la)$ factorizes as $L_{\mf l}(\la) = V_1\widehat{\otimes} \cdots \widehat{\otimes}V_k$. Here, each $V_i$   is a simple highest weight module over the $i$-th summand $\mf q(n_i)$, for   $1\leq i\leq k$.  It follows that the Backelin functor $\Gamma_\zeta^{\mf l}$ acts component-wise: \begin{align} \Gamma_\zeta^{\mf l}(L_{\mf l}(\la)) = \Gamma_\zeta^{\mf q(n_1)}(V_1)~\widehat{\otimes}~\cdots \widehat{\otimes}~\Gamma_\zeta^{\mf q(n_k)}(V_k), \label{eq::sameType} \end{align} where each $\Gamma^{\mf q(n_i)}_\zeta$ is a $\mf q(n_i)$ analogue of the Backelin functor. Furthermore,  $\Gamma_{\zeta}^{\mf l}(L_{\mf l}(\la))$ and $L_{\mf l}(\la)$ share the same type (Type \texttt{M} or \texttt{Q}). This follows from the fact that these two modules have identical annihilator ideals by \cite[Corollary 34]{CC23_2}, which, as shown in \cite[Lemma 6]{CCS25}, uniquely determine the type of a simple module over a finite-dimensional complex Lie superalgebra.

We define the queer analogue of standard Whittaker modules as follows:
\begin{align}
&M(\la,\zeta): = M^{\mf p}(\Gamma_\zeta^{\mf l}(L_{\mf l}(\la))), \text{ for }\la \in \Lambda(\nu),\label{eq::stdWhittaker}
\end{align} and let $L(\la,\zeta)$ be the unique simple quotient of $M(\la,\zeta)$. Let $\mc N(\zeta)_{\nu+\Lambda}$ be the Serre subcategory of $\mathcal{N}(\zeta)$ generated by composition factors of modules
of the form $\Gamma_\zeta(L(\la))$, for $\la\in \Lambda(\nu)$.  By the proof of \cite[Lemma 3]{CC23_2}, this category coincides with  the   Serre subcategory of $\mathcal{N}(\zeta)$ generated by simple objects $S\in \mc N(\zeta)$ whose restriction $\Res S$ has the $\g_\oa$-central character coinciding with the $\g_\oa$-central character of $L_{\g_\oa}(\la)$, with $\la \in \nu+\Lambda$, where $L_{\g_\oa}(\la)$ denotes the simple highest weight module over $\g_\oa$ of highest weight $\la$.
The following theorem justifies the above notation. 
\begin{thm} \label{eq::compostadWhi} Let $\zeta\in \ch\mf n_\oa$. 
The set $\{L(\la,\zeta)|~\la \in \Lambda(\nu)\}$ is an exhaustive list of
	mutually non-isomorphic simple Whittaker modules in $\mc N(\zeta)_{\nu+\Lambda}$.
    Furthermore, for any $\la, \mu \in \Lambda(\nu)$, we have
	\begin{align}\label{eq::13inThm}
	&[M(\la,\zeta): L(\mu, \zeta)] = [M^{\mf p}(\la): L(\mu)]. \end{align}
    \end{thm} \begin{proof}
Applying the exact functor  $\Gamma_\zeta$ to the canonical quotient $M^{\mf p}(\la)\twoheadrightarrow L(\la)$, we obtain an epimorphism $M(\la,\zeta)\twoheadrightarrow \Gamma_\zeta(L(\la))$ by Proposition \ref{prop::bacpara}. According to Lemma~\ref{lem::simpletop9}, the module $M(\la,\zeta)$ has a unique simple quotient. This, together with \cite[Theorem 33]{CC23_2}, implies  that $L(\la,\zeta)\cong \Gamma_\zeta(L(\la))$.

The first assertion now follows from \cite[Theorem 33]{CC23_2}, and Equation \eqref{eq::13inThm} follows from the exactness of the functor $\Gamma_\zeta$.
\end{proof}

Theorem \ref{eq::compostadWhi} reduces the computation of the composition multiplicities $[M(\la,\zeta):L(\mu,\zeta)]$ to that of $[M^{\mf p}(\la):L(\mu)]$, i.e., the composition multiplicities of parabolic Verma modules for the queer Lie superalgebras. They are obtained from the usual composition multiplicities of Verma modules in the BGG category of queer Lie superalgebras as follows. We have the following identitiy in the Grothendieck group of BGG category of $\mf l$-modules:
 \[
 [L_{\mf l}(\la)]=\sum_{\gamma}a_{\gamma\la}[M_{\mf l}(\gamma)].
 \]
 On the other hand, we have the following identity in the Grothendieck group of BGG category of $\g$-modules:
\[
[M(\gamma)]=\sum_{\mu}b_{\mu\gamma} [L(\mu)].
\]
Therefore, we have
\begin{align*}
   [M^{\mf p}(\la)]&=[\Ind_{\mf p}^\g L_{\mf l}(\la)]=\sum_{\gamma} a_{\gamma\la}[\Ind_{{{\mf p}}}^{{\g}} M_{\mf l}(\gamma)]= \sum_{\gamma} a_{\gamma\la}[{M}(\gamma)]\\
   &=\sum_{\gamma} a_{\gamma\la}\sum_{\mu}b_{\mu\gamma} [L(\mu)] = \sum_{\mu} \left(\sum_{\gamma}a_{\gamma\la}b_{\mu\gamma}\right)[L(\mu)].
\end{align*}
Thus, $[M^{\mf p}(\la):L(\mu)]=\sum_{\gamma}a_{\gamma\la}b_{\mu\gamma}$.
Since $\g$ is a queer Lie superalgebra and $\mf l$ is a direct sum of queer Lie superalgebras, the integers $a_{\la\mu}\in\Z$ and $b_{\mu\gamma}\in\N$ can be computed by composition multiplicities for the queer Lie superalgebras in the BGG category.

It follows from \cite{Ch16} that the composition multiplicity problem above can be reduced to the problem for three specific types of weights, namely weights of the form $s+\Z$ with $s\in\Z$, $s\not\in\hf\Z$ and $s\in\hf+\Z$. In \cite{CKW17}, it was conjectured that the composition multiplicities in the latter two cases, i.e., the case when $s\not\in\Z$, are computed by Lusztig's canonical and Webster's orthodox basis of quantum groups of type $A$ and $C$, respectively. These conjectures have been subsequently established in \cite{BD17, BD19}.

\subsection{The target category of Backelin's functor}\label{sec::CatW}
Let $\mc W_\nu(\zeta)$ denote the image  category  of $\Gamma_\zeta: \mc O_{\nu+\Lambda}\rightarrow \mc N(\zeta)$.
By \cite[Theorem 33]{CC23_2} the  functor $\Gamma_\zeta: \mc O_{\nu+\Lambda}\rightarrow \mc W_\nu(\zeta)$ satisfies the universal property of a Serre  quotient functor in the sense of \cite[Corollaries II.1.2 and III.1.3]{Ga62}. Therefore,  we have equivalences  $$\mc W_\nu(\zeta)\cong \ov{\mc O}_{\nu+\Lambda}\cong \mc O^{\nu\text{-pres}}.$$ In particular, this implies that $\mc W_\nu(\zeta)$ possesses the BGG-type reciprocity described in Proposition \ref{prop::BGGforpres}.  The goal of this subsection is to give a precise description of objects  in the Whittaker category $\mc W_\nu(\zeta)$.  Recall that the Backelin functor $\Gamma_\zeta$  (functorially) commutes with both the functors of tensoring with finite-dimensional weight modules and the induction functor $\Ind$. This implies that  $\mc W_\nu(\zeta)$ contains all  Whittaker modules of the form $E\otimes \Ind M_0(\nu, \zeta)$, where $E$ are  finite-dimensional weight $\g$-modules.

\subsubsection{Action of the center} \label{sect::actofS}
This subsection is devoted to developing a deformation of the full BGG category $\mc O$. Let $\hat{S}$ be  the completion of the symmetric algebra $S= S(\mf h_\oa)$ of $\mf h_\oa$ at the maximal ideal generated by  $\mf h_\oa$ and $W_\zeta = W_{\mf l_\zeta}$ be the Weyl group of $\mf l_\zeta$.  Consider the subalgebra $\hat{S}^{W_\zeta}$ consisting of all invariants of $\hat S$ under the action of    $W_\zeta$. 
Following the framework of Mili{\v{c}}i{\'c} and Soergel \cite[Sections 4,5]{MS97} and its generalization to the super case in  \cite[Section 4.1]{C21}, there exists a ring homomorphism $\vartheta: \hat{S}^{W_\zeta} \rightarrow \End(\text{Id}_{\mc N(\zeta)})$, where $\text{Id}_{\mc N(\zeta)}$ is the identity functor of $\mc N(\zeta)$. The construction of $\vartheta$ is recalled below.

We now construct the action of $\hat{S}^{W_\zeta}$ on $\mc N(\zeta)$ as follows.  Let $Z = Z(\mf l_\zeta)$ be the center of $U(\mf l_\zeta)$. For any central character $\chi_\la: Z\rightarrow \C$ associated with a  weight $\la\in \h^\ast_\oa$, we let $\hat{Z}_\la$ denote the completion of $Z$ at $\ker(\chi_\la)$. Let $\hat{S}_\la$ be  the completion of $S$ at the maximal ideal generated by  $\{h-\la(h)|~h\in \mf h_\oa\}$, for $\la \in \h_\oa^\ast$.  Then the Harish-Chandra homomorphism $Z\rightarrow S(\h_\oa)$ induces an inclusion $\hat Z_{\chi_\la} \hookrightarrow \hat{S}_\la$. Composing this inclusion with the translation by $\la$, we get an inclusion $  \hat Z_{\chi_\la} \hookrightarrow \hat{S}$ containing  $\hat{S}^{W_\zeta}$ as a subalgebra.

Note that $\mc N(\zeta) =\bigoplus_{\la\in \h^\ast_\oa} \mc N(\zeta)_{\chi_\la}$ decomposes into full subcategories $\mc N(\zeta)_{\chi_\la}$ consisting of objects $M$ in $\mc N(\zeta)$ on which $z-\chi_\la(z)$ acts locally nilpotently, for  $z\in Z$.  Therefore $\hat{Z}_\la$ acts on modules in $\mc N(\zeta)_{\chi_\la}$, for each $\la \in \h_\oa^\ast$, and this allows us to  define,  
for $M\in \mc N(\zeta)$,  an algebra homomorphism $\vartheta_M: \hat{S}^{W_\zeta}\rightarrow \End_{\g_\oa}(\Res M)$. 
By \cite[Proposition 11]{C21},  we have  $\vartheta_M(\hat{S}^{W_\zeta})\subseteq  \End_{\g}(M)$. Let $I_\zeta$ be the maximal ideal of $\hat{S}^{W_\zeta}$.  We define  the following full abelian subcategory of $\mc N(\zeta)$:
\[\widehat{\mc  W}(\zeta) =\{M\in \mc N(\zeta)|~\vartheta(I_\zeta)M=0\}.\]

\begin{lem} \label{eq::cathatW}
	For any $\zeta \in \ch\mf n_\oa$, the category $\widehat{\mc W}(\zeta)$ contains all both standard and simple objects in $\mc N(\zeta)$.
\end{lem}
\begin{proof} It suffices to show that $M^{\mf p}(V)\in \widehat{\mc W}(\zeta)$, for any $V\in \text{Irr}\mc N_{{\mf l}}(\zeta)$. Assume first  that $\mf l = \mf q(n)$. By \cite[Lemma 3]{CC23_2},  $M^{\mf p}(V) =V$ is a quotient of $\Gamma_\zeta({L(\la)})$, for some {$\la\in \h^\ast_\oa$}. This implies that  $M^{\mf p}(V)$ is  a quotient of $\Ind M_0(\la,\zeta)$, for some $\la \in \h^\ast_\oa$. Since $\vartheta(I_\la)$ acts on $M_0(\la,\zeta)$ as zero, the conclusion now follows by \cite[Theorem~4.1]{MS97}.
\end{proof}

\begin{ex} When $\zeta =0$, the category    $\widehat{\mc W}(0)$ consists of  weight modules in $\mc N(0)$. In this case, $\widehat{\mc W}(0)$ coincides with the full BGG category $\mc O$.
\end{ex}

\subsubsection{Description of objects in $\mc W_\nu(\zeta)$}
Fix a nilcharacter  $\zeta\in \ch\mf n_\oa$.
In the following proposition, we provide a description  of the objects in $\mc W_\nu(\zeta)$. In particular, we show that $\mc W_\nu(\zeta)$ is   the smallest full abelian subcategory of $\g\Mod$ that contains $\Ind M_0(\nu,\zeta)$ and is stable under the action of projective functors.
\begin{prop}
 Let $M \in \mc N(\zeta)$. Then the following are equivalent:
 \begin{enumerate}
 	\item[(a)] $M\in \mc W_\nu(\zeta)$.
 	\item[(b)] $M\in \widehat{\mc W}(\zeta)\bigcap \bigoplus_{\la \in \nu+\Lambda} \mc N(\zeta)_{\chi_\la}.$
 	\item[(c)] $M$ is isomorphic to a subquotient of a module of the form $E\otimes \Ind M_0(\nu,\zeta)$, for some finite-dimensional weight $\g$-module $E$.
 \end{enumerate}

\end{prop}
\begin{proof} By \cite[Theorem 33]{CC23_2}, $\mc W_\nu(\zeta)$ consists exactly of modules $M\in \mc N(\zeta)$ that admits a two-step resolution of the form $X\rightarrow Y\rightarrow M\rightarrow 0$, where $X,Y$ are   summands of modules of the form $E\otimes \Ind M_0(\nu,\zeta)$. 
	By  \cite[Theorem 4.1]{MS97} and \cite[Theorem 16]{C21},  we have the equivalence $(a)\Leftrightarrow (b)$.  This also proves that $\mc W_\nu(\zeta)$ is a full abelian subcategory of $\mc N(\zeta)$ by the construction of $\widehat{\mc W}_\nu(\zeta)$.
	
	By the above argument, we have  $(a)\Leftrightarrow (c)$. Since $E\otimes \Ind M_0(\nu,\zeta)$ is an object in the full abelian subcategory  $\mc W_\nu (\zeta)$, so are its subquotients. This completes the proof.
\end{proof}

\section{Finite $W$-superalgebras} 
In this section, we formulate the finite $W$-algebra $U(\g,E)$ for $\g= \mf q(n)$ associated with an odd nilpotent element $E\in \g_\ob$. Subsequently, we set up the Losev--Shu--Xiao decomposition and provide  descriptions of the BGG-type categories of $U(\mf g, E)$ introduced by \cite{BGK08}. We then establish several equivalences between the categories  of Whittaker $\mf q(n)$-modules and these BGK categories of $U(\mf g,E)$. Finally, by applying the results  in the prior sections, we obtain a reduction of the multiplicity problem for Verma modules for  $U(\mf g,E)$ to that of Verma modules for $\g$.
\label{sect::Walg}
\subsection{Good $\Z$-gradings} 
Let $\chi \in \g_\oa^\ast$ be a nilpotent linear functional in the coadjoint representation of the reductive algebraic group of $\g_\oa$. We can always regard $\chi \in \g^\ast$ by declaring that $\chi(\g_\ob)=0$. Let $\g^\chi$ be the centralizer of $\chi$ in $\g$, that is, \[\g^\chi = \{x\in \g|~\chi([x,\g ])=0\}.\] Recall that $\mf z(\g)$ denotes the center of $\g$. Following \cite{Zh14},    a {\em good} $\Z$-grading for  $\chi$ is a $\Z$-grading  of $\g$ $$\Gamma: \g= \bigoplus_{i\in \Z} \g(i)$$ such that   the following conditions  hold:
\begin{itemize}\item[(a)]   $\mf z(\g) \subseteq \g(0)$;
	\item[(b)] $\chi(\g(i))=0$, unless $i=-2$;
	\item[(c)] $\g^\chi \subseteq \bigoplus_{i\geq 0}\g(i)$.
\end{itemize}

Define the odd trace form
$$\text{otr}({}_-): \mf q(n)\rightarrow\C,\hskip0.2cm\left( \begin{array}{cc} A & B\\
	B & A\\
\end{array} \right)\mapsto \text{tr}(B),$$ where $\text{tr}$ denotes  the usual trace of a matrix.  This gives rise to an {\em odd} non-degenerate $\g$-invariant supersymmetric bilinear form of $\g$:
\[(x|y): = \text{otr}(xy), \text{for x,y}\in \g, \]
which allows us to identify $\g$ with $\Pi(\g^*)$. Here we recall that $\Pi$ denotes the parity-reversing map of $\g$.

Let $\chi \in \g^\ast$ be an even nilpotent linear functional. Since $\g\cong \Pi(\g^\ast)$, it follows that  there is a unique odd nilpotent element $E\in \g_\ob$ determined by $\chi(\_) = (E|\_)$. We set  $$\g^E:= \g^\chi = \{x\in \g|~[E,x]=0\},~e:=\Pi E.$$ 
Recall that $\text{I}_{2n}\in \mf \g_\oa$ denotes the identity matrix of $\mf q(n)$. By \cite[Lemmas 2.1, 2.2]{Zh14}, $\Gamma$ is a good $\Z$-grading if and only if $\text{I}_{2n}\in \g(0)$ and  the restriction  $$\Gamma|_{\g_\oa}: \g_\oa =\bigoplus_{i\in \Z} \g_\oa(i)$$ of $\Gamma$ to $\g_\oa\cong \gl(n)$ is a good grading for $e$ in the sense of \cite{EK05}. The latter is known and given in terms of certain  combinatorial objects called pyramids \cite[Section 4]{EK05}. In this case, there is a semisimple element $h_\Gamma\in [\g_\oa, \g_\oa]$ such that $\g(i) = \{x\in \g|~[h_\Gamma,x] =ix\}$ and so   $\Pi\g(i) = \g(i)$, for all $i\in \Z$.

\subsection{Finite $W$-algebras for $\mf{q}(n)$} \label{sect::finiteWq}
Fix an even nilpotent linear functional $\chi \in \g^\ast$. Define the following (even) super-skewsymmetric bilinear form $$\omega_\chi(\_,\_): = \chi([\_,\_]): \g\times \g\rightarrow \C.$$
 Recall that   $E\in \g_\ob$ denotes  the element determined by $\chi(\_) = (E|\_)$ and $e =\Pi E$. We fix a good $\Z$-grading for $\chi$ as above.  By \cite[Proposition 2.4]{Zh14}, the adjoint action $\ad E$ restricts to a bijection from $\g(-1)$ to $\g(1)$. This leads to a non-degenerate super-symplectic bilinear form $\omega_\chi|_{\g(-1)}:\g(-1)\times \g(-1)\rightarrow \C$. In particular, we have $$\dim \g(-1)_\ob =\dim \Pi \g(-1)_\oa = \dim  \g(-1)_\oa \text{, which is even.}$$ Pick a $\Z_2$-graded Lagrangian subspace $l \subseteq \g(-1)$ with respect to $\omega_\chi$ and define the subalgebra \begin{align}
 &\mf m:=   l\oplus \bigoplus_{i\leq -2}\g(i). \label{eq::defm}
 \end{align} Then  $\chi$ restricts to a character on $\mf m$. We set
 \begin{align}
 &\mf m_\chi : = \{m-\chi(m)|~m\in \mf m\}, \label{eq::mchi}
 \end{align} and let $I_\chi$ be the left ideal generated by $\mf m_\chi$. Let $Q_\chi$ be the left $U(\g)$-module $U(\g)/I_\chi$. The finite $W$-algebra associated to $\g$ and $\chi$ is defined as the associative superalgebra
 \[ U(\g, E):=\End_{\g}(Q_\chi)^{\text{opp}}\cong\left(U(\g)/I_\chi\right)^{\ad\mf m}.\]
By definition, $Q_\chi$ is a $\left(U(\g),U(\g,E)\right)$-bimodule. 
Since $\dim\g(-1)$ is even, it follows by \cite[Theorems 3.6, 3.7]{Zh14}, that the definition of the finite $W$-superalgebra $U(\g, E)$ is independent of the choice of the Lagrangian subspace $\mf l$ and the good gradings $\Gamma$ for $\chi$, up to isomorphisms. Therefore, without loss of generality, we may assume that $\Gamma$ is a {\em Dynkin} $\Z$-grading
\begin{align}
&\Gamma: \g =\bigoplus_{i\in \Z} \g(i) \label{eq::Dyngr}
\end{align}
for $\chi$, that is, $\Gamma$ coincide with the eigenspace decomposition of the adjoint action an element $h\in\h_\oa$, i.e., $\g(i) = \{x\in \g|~[h,x]=ix\}$ for all $i\in \Z$, and furthermore, we have an $\mf{sl}(2)$-triple $\{e,h,f\}$ in $\g$. Associated to $\Gamma$, we define the {\em Kazhdan grading} $\g = \bigoplus_{i\in \Z}\g^{K}(i)$ of $\g$ as the grading on $\g$ determine by the eigenvalues of $\ad h$ shifted by $2$, i.e., $\g^{K}(i) = \g(i-2)$, for any $i\in \Z$. In the remaining sections, we fix such a Dynkin grading $\Gamma$ and an $\mf{sl}(2)$-triple. We  set $$F:= \Pi f,~H:=\Pi h.$$


\subsection{Skryabin’s equivalence}
A nilpotent element $E \in \g_\ob$ is {\em principal} if its associated even linear form $\chi= (E|\_)\in \g^\ast$ is {\em regular} (i.e., $\dim \g^\chi_\oa=\dim \g^E_\oa$ is minimal). For such $E$, the subalgebra $\mf m$ is the nilradical of a Borel subalgebra of $\mf g$, and   $U(\g, E)$ is referred to as a {\em principal finite $W$-algebra}, in analogy with the reductive and basic classical cases.

In the reductive setting (i.e., $\g=\g_\oa$),  Kostant \cite{Ko78}  established an   equivalence between the category $\mc N(\zeta)$ with regular $\zeta :=\chi|_{\mf m}$ and the category of finite-dimensional modules over the corresponding principal finite $W$-algebra, which is isomorphic to $Z(\g)$. This was subsequently generalized by Skryabin \cite{Skr} to the setting of arbitrary finite $W$-algebras. An analogue of Skryabin’s equivalence also holds for basic classical and queer Lie superalgebras; see \cite{Zh14,SX20}.  

\begin{lem}(Skryabin-type equivalence) \label{lem::skr}
 The Skryabin functor ${\rm Sk}(\_):= Q_\chi \otimes_{U(\g,E)}\_$ gives rise to an equivalence  from the category of $U(\g,E)$-modules to the category of $\g$-modules $M$ such that $\mf m_\chi$ acts on $M$ locally nilpotently. The   quasi-inverse of ${\rm Sk}$ is given by the Whittaker functor ${\rm Wh(\_)}$, i.e., for a $\g$-module $M$, we have
 \[{\rm Wh(M)}: =\{x\in M|~{\mf m}_\chi x=0\},\]
 which is a $U(\g,E)$-module.
 \end{lem}

\begin{rem}\label{rem::clsPS16}
The construction of the finite W-algebras for the queer Lie superalgebra $\mf q(n)$ presented above can naturally be generalized to the case of a finite direct sum of queer Lie superalgebras. Furthermore, in this general setting, Skryabin equivalence remains valid. A classification of the irreducible representations of the principal finite W-algebras for $\mf q(n)$ has been established in the work of Poletaeva and Serganova \cite{PS16b}. Combining this with Skryabin equivalence and Proposition \ref{prop::LeviSimple}, we obtain a complete classification and description of the simple objects in $\mc N_{\mf l}(\zeta)$. 
\end{rem}

In the following subsections, we explore an extension of Skryabin’s equivalence by building upon Losev's approach based on his decomposition theorem \cite{Los12}. This extension establishes equivalences between the BGK  categories $\mc O$ of finite $W$-algebras and MMS-type Whittaker categories.

\subsection{Super Darboux-Weinstein theorem}
This subsection is devoted to a super Darboux-Weinstein decomposition for $\mf q(n)$. We start with the following lemma, which is the statement of \cite[(3.1)]{Zh14}.
\begin{lem} \label{lem::1}
 We have
 \[\g = \g^E\oplus \Pi [\mf g, F] = \Pi \g^F\oplus  [\mf g, E]. \]
\end{lem}
\begin{proof} 
As observed in Part (iii) in the proof of \cite[Lemma 3.2]{Zh14}, the linear operators 
\[
    \Pi \circ \operatorname{ad} E, \quad 
    \Pi \circ \operatorname{ad} F, \quad 
    \Pi \circ \operatorname{ad} H : \mathfrak{g}_i \to \mathfrak{g}_i
\] constitute an $\mf{sl}(2)$-triple in $\End_\C(\g_i)$, for $i\in\Z_2$. The claimed decompositions in the statement of the lemma now follow from $\mf{sl}(2)$-representation theory and the fact that $\Pi^2$ is the identity function on $\g$.
\end{proof}


\begin{lem} \label{lem::2}
	The bilinear form $\omega_\chi$ restricts to a non-degenerate 
super-symplectic bilinear form  on $\Pi [\g, F]$.
\end{lem}
\begin{proof} Suppose  that $a= \Pi [x, F]$  lies in the radical of $\omega_\chi|_{\Pi [\g, F]}$, for some $x\in \g$. It follows by Lemma \ref{lem::1} that  \[ \omega_\chi(\g,a) = \omega_\chi(\Pi [\mf g, F],a)= 0.\]  In particular, we have
	$([\g,E]|a)=0,$ which, combined with Lemma \ref{lem::1} again, gives $$(\g|a) = (\Pi\g^F+[\g,E]|a)=(\Pi \g^F| a) = (\Pi \g^F| ~\Pi [x,F])= (\g^F| [x,F]) =0.$$ Therefore, $a =0$ and the conclusion follows.
\end{proof}

 	\subsubsection{Parabolic decompositions associated with $\mf t$} \label{sect::46}
 	Let $\h$ be a Cartan subalgebra of $\g$ containing $h$. Define the following subalgebra of $\g$  $$\mf t: = \{x\in \h_\oa|~[x,E]=0\} \subseteq \g_\oa(0).$$
 	Since  $\mf t$ is   an abelian subalgebra preserving the super-symplectic bilinear form $\omega_{\chi}|_{\g(-1)}$, we may choose the Lagrangian subspace $l$ of $\g(-1)$ with respect to $\omega_\chi$ from Subsection \ref{sect::finiteWq} to be $\mf t$-invariant. Hence, the subalgebra $\mf m $ is $\mf t$-stable, and it follows that
 	\[\mf t\subseteq U(\g, E).\]

 	Let $T$ be the adjoint group of $\mf t$. We pick an integral element $\theta\in \mf t$, namely, $\theta$ lies in the cocharacter of $T$. We have an $\ad \theta$-eigenspace decomposition of $\g$:
 	\begin{align}
 		&\g =\bigoplus_{i\in\Z} \g_{\theta, i},
 	\end{align} where $\g_{\theta, i} = \{x\in \g\mid [\theta, x] =ix\}$ for $i\in \Z$. Recall that  $\Phi\subseteq \mf h^\ast_\oa$ denotes the set of all roots of $\g$ and $\g^{\alpha}$ denotes   the root space for $\alpha \in \Phi$, respectively. The element $\theta$ induces a parabolic decomposition:
 	\begin{align}
 		&\mf g = \mf u^-\oplus \mf l \oplus \mf u,~\text{ where }\mf u^- =\bigoplus_{\alpha(\theta)<0} \g^{{\alpha}},~\mf l =\bigoplus_{\alpha(\theta)=0} \g^{{\alpha}}, \text{ and } \mf u =\bigoplus_{\alpha(\theta)>0} \g^{{\alpha}}. \label{eq::parade}
 	\end{align}
 	We choose a triangular decomposition $\g =\mf n^-\oplus \mf h \oplus \mf n$ to be compatible with \eqref{eq::parade}, i.e., $\mf n^-\supseteq \mf u^-$ and $\mf n \supseteq \mf u$. This defines the parabolic 
subalgebra $\mf p =\mf u \oplus \mf l$.

 \subsubsection{Super Darboux--Weinstein theorem}
  For a vector superspace $\mf a$, let $S[\mf a]^\wedge_\eta$ denote  the completion  of the symmetric superalgebra $S[\mf a]\cong \C[\mf a^\ast]$ at a point $\eta \in \mf a_\oa^\ast\subseteq \mf a^\ast$. Specifically,  $S[\mf a]^\wedge_\eta = \varprojlim S[\mf a]/ I_\eta^k$
 is the topological completion of the polynomial algebra $\C[\mf a^\ast]$ with respect to the $I_\eta$-adic topology, where $I_\eta$ is the maximal ideal of $\C[\mf a^\ast]$ corresponding to $\eta$.

 We recall the classical $\C^{\times}$-action on $\g$ that stabilizes the Slodowy slice in $\g_\oa$ associated with the nilpotent element $e=\Pi E$.  Let $G$ denote the adjoint group of $\g_\oa$ (i.e.,  the smallest algebraic
 subgroup of ${\rm GL}(\g_\oa)$ with Lie algebra  $\ad \g_\oa$).
  The $\mf{sl}(2)$-embedding $\langle e,h,f \rangle\hookrightarrow \mf g_\oa$ exponentiates to a homomorphism ${\rm SL}(2)\rightarrow G$. Its restriction to the $1$-dimensional torus consisting of diagonal matrices gives rise to an
 one-parameter subgroup $\gamma': \C^{\times}\rightarrow G$ such that the differential of $\gamma'$ at $1$ is $h$. Extend the adjoint action of $G$ to $\g$ so that   $\gamma'(t)x =t^{i}x$, for $x\in \g(i)$ and $t\in \C^{\times}$. We define the   {\em Kazhdan action} of $\C^\times$ on $\g$:
    \begin{align} \label{eq::kazact}
    &\gamma: \C^\times\times \g\rightarrow \g,~\gamma(t)x = t^2\gamma'(t)x,~\text{ for }t\in \C^\times, x\in \g.
    \end{align}

 	\subsubsection{Star products} Let $A$ be a Poisson superalgebra with the Poisson bracket $\{,\}$ and define $A[[\hbar]] := A\otimes \C[[\hbar]]$, where $\hbar$ is a formal parameter.  Suppose that  $\ast: A\otimes A \rightarrow A[[\hbar]]$ is  an associative product, that is, $(p\ast q)\ast r =p\ast (q\ast r)$, for all $p,r,q\in A$. We write $p\ast q = \sum_{i=0} D_i(p,q)\hbar^{2i}$ with $D_i(p,q)\in A$,  for $p,q\in A$. Recall that $\ast$ is called a {\em star product} if it satisfies among others the following two conditions that are relevant for us in the sequel:
 	\begin{itemize}
 		\item[(1)] the natural $\C[[\hbar]]$-extension of $\ast$ to $A[[\hbar]]\otimes_{\C[[\hbar]]}A[[\hbar]]$ is associative;
 		\item[(2)] $f\ast g -fg\in \hbar^2A[[\hbar]]$ and $f\ast g -(-1)^{\ov f\ov g}g\ast f -\hbar^2\{f,g\}\in \hbar^4A[[\hbar]]$, for all homogeneous $f,g\in A$.
 	\end{itemize}
 	We shall not give all the conditions here (see \cite[Section 2]{Los12} for details), but instead only remark that these conditions enable us to extend the $\ast$ to completions of $A[[\hbar]]$ discussed below. The star product $\ast$ is called differential if each $\text{D}_i$ is a bidifferential operator of order at most $i$ in each variable.
 	
 	Following \cite{Los10b}, We refer to $A[[\hbar]]$ as a {\em quantum algebra} when we consider it as an algebra with respect to the star product defined above. When the subspace $A[\hbar]$ is a subalgebra in $A[[\hbar]]$ with respect to $\ast$, then we call it a quantum algebra as well.
 	
 	 	Suppose $\C^{\times}$ acts on $A$ by automorphisms. We denote the action of $t$ at $a$ by $t.a$ for $t\in \C^{\times}$ and $a\in A$. We define a $\C^{\times}$-action on $A[[\hbar]]$ by
 	 	\begin{align} \label{eq::Cactcomp}
 	 	&t\cdot(\sum_{i=0}^\infty a_i \hbar^i):= \sum_{i=0}^\infty t^i (t.a_i) \hbar^i,\quad t\in\C^\times, a_i\in A. 	 	\end{align}
        We let $A[[\hbar]]_{\C^{\times}{\rm\text{-}fin}}$ denote the $\C^{\times}$-finite part of  $A[[\hbar]]$, that is, it is the sum of all finite-dimensional $\C^\times$-submodule of $A[[\hbar]]$.

 \subsubsection{Homogeneous Weyl algebra and its completion}
 It follows by a dual version of \cite[Proposition 2.4]{Zh14} that $\ad F: \g(i)\rightarrow \g(i-2)$ is surjective for any  $i\leq 1$. This, together with the fact that $\Pi \g(i) =\g(i)$ for any $i\in \Z$, implies that  $$\mf m\subseteq  \bigoplus_{i\leq -1} \g(i) \subseteq \Pi [\g, F].$$ Therefore,  by Lemma \ref{lem::2} we may conclude that  $\mf m$ is a Lagrangian subspace of $\Pi [\g, F]$ with respect to $\omega_\chi$. Let $\mf m^\ast$ be the dual of $\mf m$ with respect to $\omega_\chi$.  We define $$V :=\mf m_\chi\oplus \mf m^\ast.$$ Then the symplectic superspace $(V,\omega_\chi)$ gives rise to a natural  Poisson bracket $\{\_,\_\}$ of the symmetric superalgebra  $S(V)$  determined by  \[\{v,w\} = \omega_{\chi}(v,w),~\text{ for any }v,w \in V.\]
 This allows us to define the quantum algebra structure on $S(V)[\hbar]$ via the {\em Moyal-Weyl star product} $\ast$ (see, e.g., \cite[Example 3.2.3]{Los10b}), which is determined by
 \[v\ast w -(-1)^{\ov v\, \ov w}w\ast v = \omega_{\chi}(v,w)\hbar^2, \]
 for homogenous $v,w \in V$. Here $\ov v, \ov w$ denote the parities of $v, w$, respectively. This algebra is usually referred to as the {\em homogeneous Weyl algebra}. We note that quotient $S(V)[\hbar]/(\hbar -1)$ is the usual Weyl algebra. Since the Moyal-Weyl star product is differential, we can extend it to the completion $S(V)^\wedge_0[[\hbar]]$ of $S(V)[\hbar]$; see \cite[Section~2]{Los10b} and \cite[Subsection 2.4]{Los11}. We equip $V$ with the Kazhdan action of $\C^\times$ in \eqref{eq::kazact}, that is,  $t.v = \gamma(t)v$, for $t\in \C^\times$ and $v\in V$.  Recall that $V$ is a $\mf t$-module, and so it is equipped with a natural action of the adjoint group  $T$ of $\mf t$.  Therefore, the group  $T\times \C^{\times}$ acts on the quantum algebra $S(V)^\wedge_0[[\hbar]]$.

  \subsubsection{   Gutt star product} 
  The symmetric algebra $S(\g)$ of $\g$ admits a natural Poisson bracket $\{\_,\_\}$ determined by  $$\{x,y\}= [x,y],\hskip0.2cm \text{ for any }x,y\in \g.$$ Set $T(\g)$ to be the tensor algebra of $\g$.
   Then the quantum algebra  $(S(\g)[[\hbar]],\ast)$ via the {\em Gutt star product} $\ast$ can be identified as the quotient algebra of $T(\g)[[\hbar]]$ by the ideal generated by $x\otimes y -(-1)^{\ov x~\ov y}y\otimes x -[x,y]\hbar^2$, for all $x,y\in \g$ homogenous.  We note that  the quotient algebra $S(\g)[\hbar]/(\hbar -1)$ is isomorphic to $U(\g)$.    We can define the quantum algebra $(S(\g^E)[[\hbar]],\ast)$. Similarly, we equip both $S(\g)[[\hbar]]$ and $S(\g^E)[[\hbar]]$ with the natural adjoint action of $T$ and the Kazhdan action of $\C^\times$. We extend the action of $T\times \C^{\times}$ to their completions $S(\mf g^E)^\wedge_\chi[[\hbar]]$ and $S(V)^\wedge_0[[\hbar]]$.

 \subsubsection{Equivariant super Darboux-Weinstein theorem}
 Let  $S(\g^E)_\chi^\wedge\widehat{\otimes} S(V)_0^\wedge$ denote the completion of the tensor product $S(\g^E){\otimes} S(V)  =S(\g^E\times V)$ at the point $(\chi,0)\in (\g^E\times V)^\ast$. We have the following equivariant version of the super Darboux-Weinstein theorem \cite[Theorem 1.3]{SX20}, which is a super generalization of \cite[Theorem 3.3.1]{Los10b}. We state it in the slightly more general form as in \cite[Proposition]{CW25}.
 
  \begin{thm}We have a $T\times \C^{\times}$-equivariant  isomorphism of Poisson superalgebras
  \[S(\g)^\wedge_\chi \cong S(\g^E)_\chi^\wedge\widehat{\otimes}  S(V)_0^\wedge.\]
  \end{thm}

Here and in the sequel the notation $\widehat{\otimes}$ denotes the completion of the tensor product with respect to the maximal ideal $(\chi,0)$ (see,e.g., \cite[Remark 6.2]{CW25}).

 	\subsection{Losev-Shu-Xiao type decomposition}\label{sect::LSX::decom}
  The following decomposition theorem can be derived by following the strategy laid out in \cite[Theorem 6.5]{CW25} (cf. \cite[Theorem 1.6]{SX20}, \cite[Proposition 2.1]{Los12}).

  \begin{thm} \label{thm::Losevdecom}
  	We have a $T\times \C^\times$-equivariant $\C[[\hbar]]$-linear isomorphism of the quantum algebras:
  	\begin{align}
  		&\Phi_{\hbar}: S(\g)^{\wedge}_\chi[[\hbar]] \cong S(\mf g^E)^\wedge_\chi[[\hbar]]\widehat{\otimes}_{\C[[\hbar]]}S(V)^\wedge_0[[\hbar]]. \label{eq::Losevdecom}
  	\end{align}
  \end{thm}  Again, $\widehat{\otimes}$ denotes the completed tensor product of the topological algebras  $S(\mf g^E)^\wedge_\chi[[\hbar]]$ and $S(V)^\wedge_0[[\hbar]]$, with respect to the induced topology of $S(\mf g^E)^\wedge_\chi[[\hbar]]\otimes S(V)^\wedge_0[[\hbar]]$.

This type of decomposition theorem was initially established by Losev in \cite[Theorem 3.3.1]{Los10b} for reductive Lie algebras, then extended to the setting of basic Lie superalgebras by Shu and Xiao in \cite[Theorem 1.6]{SX20}, where it was formulated as a $\C^\times$-equivariant version only. In \cite[Theorem 6.5]{CW25}, it was observed that the arguments in \cite{SX20} give a slightly stronger $T\times \C^\times$-equivariant version.

  \begin{rem} \label{rem::23}
  The isomorphism restricts to an isomorphism between the subalgebras  $S(\g)^{\wedge}_\chi[[\hbar]]_{\C^\times{\text{-\rm fin}}}$ and $S(\mf g^E)^\wedge_\chi[[\hbar]]_{\C^\times{\text{-\rm fin}}}\widehat{\otimes}_{\C[[\hbar]]}S(V)^\wedge_0[[\hbar]]_{\C^\times{\text{-\rm fin}}}$ of the respective $\C^\times$-finite parts. For example, since $\g^E\subseteq \bigoplus_{i\geq 0}\g(i)$ has positive Kazhdan grading, the $\C^\times$-finite part of $S(\g^E)^\wedge_\chi[[\hbar]]$ coincides with the quantum subalgebra $\mc W_{\hbar} := (S(\g^E)[\hbar]$, $\ast$).  These algebras admit {\em Rees algebra} realizations, by Lemma \cite[Lemma 3.2.5]{Los10b}; see also  \cite[Lemma 3.4]{SX20}. We set $$\mc W: = \mc W_{\hbar}/(\hbar-1)\mc W_{\hbar}$$ to be the quotient of $(S(\g^E)[\hbar], \ast)$ by setting $\hbar =1$. By the same argument as in the proof of \cite[Corollary 3.3.3]{Los10b} (see also \cite[Theorem 3.8]{SX20}) together with an analogue of \cite[Lemma 3.7]{SX20}, it follows that   $\mc W$ and $U(\g,E)$ are isomorphic as both associative algebras and $\mf t$-modules.
    \end{rem}

  As demonstrated in \cite[Example 6.6]{CW25}, Theorem \ref{thm::Losevdecom} implies Lemma \ref{lem::skr}  when the Dynkin grading in \eqref{eq::Dyngr} is even. We now present an analogue of \cite[Theorem 6.7]{CW25}, adapted to our  $\mf q(n)$ setting.

Following \cite[Section~4]{Los12}, we let $d$ be the maximal eigenvalue of $\ad h$ on $\g$, and let $m$ be a positive integer such that $m>2d+2$.
Recall the element $\theta\in \mf t$ and the parabolic decomposition $\g=\mf u\oplus \mf l \oplus \mf u^-$ from Subsection~\ref{sect::46}. Consider the following new $\Z$-gradation of $\g$:
   \begin{align}
   	&\g = \bigoplus_{i\in \Z}\g'(i),~\text{where}~\g'(i) =\{x\in \g|~[h-m\theta, x] =(i-2)x\}. \label{eq::newgr}
   \end{align} 	Since $\mf l$ is a direct sum of queer Lie superalgebras, we may form the Levi analogue $\underline{\mf m}$ of $\mf m$ in \eqref{eq::defm} and $\underline{\mf m}_\chi$ of $\mf m_\chi$ in \eqref{eq::mchi}, respectively.    Define \begin{align}
       &\widetilde{\mf m}: = \underline{\mf m} +\mf u,~\text{ and ~ }\widetilde{\mf m}_\chi:= \underline{\mf m}_\chi+\mf u. \label{eq::23::m}
   \end{align}  Since
   $\g'(0) \subseteq  \underline{\mf m}\subseteq \oplus_{i\leq 0} \mf g'(i)$, $\mf u\subseteq \oplus_{i< 0} \mf g'(i)$ and  $\mf u^-\subseteq \oplus_{i> 0} \mf g'(i)$, it follows that $\widetilde{\mf m} = \bigoplus_{i\leq 0}\g'(i).$ We note that $\widetilde{\mf m}\cap V$ is a Lagrangian subspace of $V$ with respect to $\omega_\chi$.  Let ${\mathbf A}(V)$ denote the Weyl superalgebra of the symplectic superspace $V = (\widetilde{\mf m}\cap V) \oplus (\widetilde{\mf m}\cap V)^\ast$.

   Let $\varphi: \C^\times \hookrightarrow \C \times T$ be the diagonal embedding, whose differential is $(1,-m\theta)\in \C \times \mf t$, that is, the pull-back $\C^\times$-action induces the grading in  \eqref{eq::newgr}. By adapting  \cite[Lemma 3.7, Theorem 3.8]{SX20} to our new  $\C^\times$-equivariant setting, we can extend the algebra isomorphism of the $\C^\times$-finite parts by  \cite[Proposition 4.1]{Los12}. The following is an analogue of \cite[Theorem 6.7]{CW25} in our setting:
   \begin{thm} \label{thm::LSXCW}
   	The restriction of isomorphism $\varphi$ in \eqref{eq::Losevdecom}, with respect to the new $\C^\times$-action, to the $\C^\times$-finite parts extends to an isomorphism of algebras:
   	\begin{align}
   	&U(\g)_{\widetilde{\mf m}}^\wedge \cong U(\g,E)^\wedge_{\widetilde{\mf m}\cap \g^E} \otimes   {\bf A}(V)^\wedge_{\widetilde{\mf m}\cap V}, \label{eq::thm::LSXCW}	\end{align} where the notation ${\cdot}^\wedge_{\widetilde{\mf m}}$ denotes completion with respect to $\widetilde{\mf m}$.
   \end{thm}

  \subsection{Generalized Whittaker modules}  In this subsection, we assume that the nilpotent element $E$ is principal in $\mf l$.   This implies that the integral element $\theta$ is regular in $\mf t$.  
  In this case, all irreducible $U(\mf l,E)$-modules are finite dimensional; see also \cite{PS16}.

  A finitely-generated $U(\g)$-module $M$ is called a generalized Whittaker module corresponding to $\theta$ and $E$, provided that the subalgebra $\widetilde{\mf m}_{\chi}$ acts locally nilpotently on $M$. We denote by $\widetilde{\mf{Wh}}(\theta,E)$ the category of generalized Whittaker modules over $\g$; see also \cite{Los12, CW25}.

  Let ${\rm Sk}^{\mf l}(\_)$ be the Levi analogue of the Skryabin functor ${\rm Sk}$ in Lemma \ref{lem::skr}. Recall the parabolic subalgebra $\mf p =\mf u\oplus \mf l$ from Subsection~\ref{sect::46}. Composing with the parabolic induction functor $M^{\mf p}$ from \eqref{def::paraind},  the functor $$M^{\theta, E}:=M^{\mf p}\circ {\rm Sk}^{\mf l}$$ restricts to an exact functor from the category of finitely-generated $U(\mf l,E)$-modules to $\widetilde{\mf{Wh}}(\theta, E)$. This functor admits a right adjoint functor $\mc G(\_):=(\_)^{\widetilde{\mf m}_\chi}$ of taking $\widetilde{\mf m}_\chi$-invariants of generalized Whittaker modules. Denote by ${\mf{Wh}(\theta, E)}$ the full subcategory of $\widetilde{\mf{Wh}}(\theta, E)$ of modules $M$ such that $\dim M^{\widetilde{\mf m}_\chi}<\infty$.

 Since the nilpotent element  $E$ is principal in $\mf l$, it follows that the subalgebra $\widetilde{\mf m}$ is the nilradical of a Borel subalgebra $\mf b = \mf h\oplus \widetilde{\mf m}$ of $\g$. We define a nil-character $\zeta \in \text{ch} \widetilde{\mf m}_\oa$ by declaring that $\zeta|_{\mf u}=0$ and $\zeta|_{\underline{\mf m}}:= (E|\_): \underline{\mf m}\rightarrow \C$, that is, $\zeta = (E|\_) :\widetilde{\mf m} \rightarrow \C$. Then the corresponding  McDowell-Mili{\v{c}}i{\'c}-Soergel type  Whittaker category $\mc N(\zeta)$ is a full subcategory of $\widetilde{\mf{Wh}}(\theta, E)$, and so is the category $\widehat{\mc W}(\zeta)$ from Subsection \ref{sect::actofS}. The following lemma identifies $\mc N(\zeta)$ with two other significant categories: the  full subcategory of finite length objects in $\widetilde{\mf{Wh}}(\theta, E)$ and the category of finite-type generalized Whittaker modules in the sense of \cite[Section 3]{Los12}.

  \begin{lem} \label{lem::22} Let $M\in \widetilde{\mf{Wh}}(\theta, E)$. Then the following conditions are equivalent:
  	\begin{itemize}
  	\item[(1)]  $\dim M^{(\widetilde{\mf m}_\chi)_\oa}<\infty$.
  	\item[(2)] $M\in \mc N(\zeta)$.
  	\item[(3)] $M$ has a (finite) composition series.
  	\end{itemize}
  \end{lem}
  \begin{proof}
  	The equivalence of the assertions (1) and (2) can be proved
  	using analogous arguments as \cite[Lemma 6.10]{CW25}. The implication $(2) \Rightarrow (3)$ follows from \cite[Theorem 2.6]{MS97}.  It remains to show the    implication $(3) \Rightarrow (2)$. First, suppose that $M$ is simple. In this case,  the restriction $\Res M$ of $M$ is locally finite over $Z(\g_\oa)$  by \cite[Proposition 1]{C21}, and so we have  $M\in \mc N(\zeta)$. Finally, it follows by \cite[Lemma 12]{CC23_2} that $\mc N(\zeta)$ is a Serre subcategory of $\g\Mod$, and so it is closed under extensions in $\widetilde{\mf{Wh}}(\theta, E)$. This completes the proof.  \end{proof}

  \subsection{Category $\mc O$ of finite $W$-superalgebras}
In this subsection, we continue to assume that $E$ is an odd nilpotent element of $\g$ and $\theta\in\mf t$ is regular, so that $E$ is principal nilpotent in $\mf l$.

 Recall that $U(\g,E)$ and $U(\g^E)$ are isomorphic as $\mf t$-modules. We have a decomposition of $U(\g,E)$ into its $\ad\theta$-eigenspaces: $U(\g,E) = \bigoplus_{i\in \Z}U(\g, E)_{i}$. We set
 \[U(\g,E)_{\geq 0}:= \bigoplus_{i\geq 0}U(\g,E)_i,\hskip0.2cm~U(\g,E)_{> 0}:= \bigoplus_{i> 0}U(\g,E)_i.\] Furthermore, we define
 \[U(\g,E)_\sharp: = U(\g,E)_{\geq 0}\cap U(\g,E)U(\g,E)_{>0}.\] Then $U(\g,E)_{\geq 0}$ is a subalgebra of $U(\g,E)$ and both $U(\g,E)_{>0}$ and $U(\g,E)_\sharp$ are two-sided ideals of $U(\g,E)_{\geq 0}.$ Applying the arguments in \cite[Remark 6.8]{CW25} in conjunction with Theorem \ref{thm::Losevdecom},  we have
  \begin{align}
  &U(\mf l,E)\cong U(\mf g, E)_{\geq 0}/U(\g,E)_{\sharp}.\label{eq::Wind}
  \end{align}

 We let $\widetilde{\mc O}({\theta},E)$ denote the category of finitely-generated $U(\g,E)$-modules $M$ such that for any $m\in M$ there exists $i_m\in \Z$ with $U(\g, E)_k m =0$ for all $k\geq i_m$.
For any $U(\mf l,E)$-module $V$, we  may extend $V$  to a $U(\g, E)_{\geq 0}$-module by letting $U(\mf g,E)_\sharp$ act on $V$ trivially and define the following induced module:
  \[\mc M^{\theta,E}(V):= U(\mf g,E)\otimes_{U(\mf g,E)_{\geq 0}}V.\]
   This defines a right exact functor $\mc M^{\theta, E}$ from the category of finitely-generated $U(\mf l,E)$-modules to $\widetilde{\mc O}({\theta},E)$.
 This functor has a right adjoint functor $\mc F$ given by the $\mc F(M):= M^{U(\g,E)_{>0}}$  functor of taking $U(\g,E)_{>0}$-invariants of modules $M\in \widetilde{\mc O}({\theta},E)$. When $V$ is an irreducible $U(\mf l,E)$-module, we refer to $\mc M^{\theta, E}(V)$ as the {\em Verma module over $U(\g,E)$}.  We set $\mc O({\theta},E)$ to be the full subcategory of $\widetilde{\mc O}({\theta},E)$ of $U(\g,E)$-modules for which $\dim \mc F(M)<\infty.$ Finally, we denote by ${\mc O}({\theta},E)^{\rm fin}$ the full subcategory of ${\widetilde{\mc O}}({\theta},E)$ of all finite-length modules in ${\widetilde{\mc O}}({\theta},E)$.

 Consider the categories of (topological) modules equipped with discrete topologies over the algebras on both sides of  \eqref{eq::thm::LSXCW}. Then, the  category of  $U(\g)_{\widetilde{\mf m}}^\wedge$-modules can be identified as $\widetilde{\mf{Wh}}(\theta, E)$. If we let $\widetilde{\mf{Wh}}'(\theta, E)$ denote the category of (topological) modules over $U(\g,E)^\wedge_{\widetilde{\mf m}\cap \g^E} \otimes   {\bf A}(V)^\wedge_{\widetilde{\mf m}\cap V}$, then we obtain an equivalence of categories $ \widetilde{\mf{Wh}}(\theta, E)\cong  \widetilde{\mf{Wh}}'(\theta, E)$ by applying \cite[Lemmas 3.2.8 and 3.2.9]{Los10b} as in \cite[Proposition 5.1]{Los12}.  Combined with the equivalence $\widetilde{\mf{Wh}}'(\theta, E)\rightarrow \widetilde{\mc O}(\theta, E),~M\mapsto M^{\widetilde{\mf m}\cap V}$, we have an equivalence of categories
  \begin{align}
  &\mc K: \widetilde{\mf{Wh}}(\theta, E) \xrightarrow{\cong} \widetilde{\mc O}(\theta, E),
\end{align} which restricts to an equivalence   $\mf{Wh}(\theta, E) \xrightarrow{\cong} \mc O(\theta, E)$. Furthermore, it follows by Lemma \ref{lem::22} that the restriction of $\mc K$ gives rise to an equivalence  $\mc N(\zeta)\cong \mc O(\theta, E)^{\rm fin}$.
  Let $\text{Id}_{\mc O(\theta, E)^{\rm fin}}$ be the identity functor on $\mc O(\theta, E)^{\rm fin}$. We obtain an isomorphism
  $$\End_{\mc N(\zeta)}(\text{Id}_{\mc N(\zeta)})\cong \End_{\mc O(\theta, E)^{\rm fin}}(\text{Id}_{\mc O(\theta, E)^{\rm fin}}).$$

  Composing with the ring homomorphism  $\vartheta: \hat{S}^{W_\zeta}\rightarrow \End_{\mc N(\zeta)}(\text{Id}_{\mc N(\zeta)})$ from Subsection \ref{sect::actofS}, we obtain an action of  $\hat{S}^{W_\zeta}$ on $\mc O(\theta, E)^{\rm fin}$. Let $\mc O^0(\theta, E)^{\rm fin}$ denote the full abelian subcategory of $\mc O(\theta, E)^{\rm fin}$ consisting of  objects  which are annihilated by the maximal ideal $I_\zeta$ of $\hat{S}^{W_\zeta}$:
 \[  \mc O^0(\theta, E)^{\rm fin}:= \{M\in  \mc O(\theta, E)^{\rm fin}|~I_\zeta M=0\}.\]

  For any irreducible $U(\mf l,E)$-module $V$, we let ${\mc  L}^{\theta, E}(V)$ denote the irreducible quotient of the $U(\g, E)$-Verma module  ${\mc M}^{\theta, E}(V)$. Recall that $\nu \in \mf h_\oa^\ast$ is a dominant weight such that   under the dot-action of $W$ its stabilizer subgroup is $W_{\mf l_\zeta}$. Denote by   $L_{\mf l}(\la,\zeta)$ the Levi analogue of  $L(\la,\zeta)$ in Subsection \ref{sect::bac} and let ${\rm Wh^{\mf l}}$ be the quasi-inverse of ${\rm Sk}^{\mf l}$. We define the $U(\g,E)$-highest weight modules $$\mc M^{\theta, E}(\la) : = \mc M^{\theta, E}({\rm Wh}^{\mf l}L_{\mf l}(\la,\zeta)),~\mc L^{\theta, E}(\la) : = \mc L^{\theta, E}({\rm Wh}^{\mf l}L_{\mf l}(\la,\zeta)),$$ for $\la \in \Lambda(\nu)$. The following is the main result in this section:

  \begin{thm} \label{thm::equivWhiWmod} There is the following commutative diagram:
  	\begin{align}
    \begin{split}
  		&\xymatrixcolsep{2pc} \xymatrix{
  	 \widehat{\mc W}(\zeta) \hskip 0.2cm \ar@{^(->}[r]  \ar[d]^{\cong}	&	\mc N(\zeta)  \hskip 0.2cm  \ar@{^(->}[r]   \ar[d]^{\cong}    &  \mf{Wh}(\theta, E) \ar@{^(->}[r] \ar@<-2pt>[d]^{\cong}   & \widetilde{\mf{Wh}}(\theta, E) \ar[d]_{\mc K}^{\cong} \\
  	 \mc O^0(\theta, E)^{\rm fin} \ar@{^(->}[r] \hskip 0.2cm & {\mc O}(\theta, E)^{\rm fin}  \ar@{^(->}[r] \hskip 0.2cm  &   \mc O(\theta, E) \ar@{^(->}[r]  \hskip 0.2cm & \widetilde{\mc O}(\theta, E)} \label{eq::8}
     \end{split}
  	\end{align}
  Furthermore, we have the following.
  \begin{enumerate}
  \item[(1)] ${\mc  M}^{\theta, E}(V), {\mc  L}^{\theta, E}(V)\in \mc O^0(\theta, E)^{\rm fin}$, for  any irreducible $U(\mf l,E)$-module $V$.
  \item[(2)] Any simple object in $\widetilde{\mc O}(\theta,E)$ is isomorphic to ${\mc  L}^{\theta, E}(V)$, for some irreducible $U(\mf l,E)$-module $V$.
  \item[(3)] For any $\la, \nu \in \Lambda(\nu)$, we have $$[{\mc  M}^{\theta, E}(\la): {\mc  L}^{\theta, E}(\mu)] = [M^{\mf p}(\la): L(\mu)].$$
  \end{enumerate}
  \end{thm}
\begin{proof} The equivalence of categories $\mc K: \widetilde{\mf{Wh}}(\theta, E) \xrightarrow{\cong} \widetilde{\mc O}(\theta, E)$ restricts to an equivalence   $\mc N(\zeta)\cong \mc O(\theta, E)^{\rm fin}$  by Lemma \ref{lem::22} and an equivalence $\mf{Wh}(\theta, E) \xrightarrow{\cong} \mc O(\theta, E)$, respectively. We have $\mf{Wh}(\theta, E)= \{M\in \widetilde{\mf{Wh}}(\theta, E)|~ \dim M^{\widetilde{\mf m}_\chi}<\infty\}$. By Lemma \ref{lem::22}, we also have $\mc N(\zeta) = \{M\in \widetilde{\mf{Wh}}(\theta, E)|~\dim M^{(\widetilde{\mf m}_\chi)_\oa}<\infty \}$. It follows that $\mc N(\zeta)\subseteq \mf{Wh}(\theta, E)$, and hence we have the inclusion $\mc O(\theta,E)^{\text{fin}}\subseteq\mc O(\theta,E)$, which in turn  completes the proof of the diagram in \eqref{eq::8}.

	It remains to show the assertions in (1)--(3).  The arguments in \cite[Corollary 6.14]{CW25} can be adapted to our situation as follows. 	For any generalized Whittaker module $M\in \widetilde{\mf{Wh}}(\theta, E),$
 we have $$ \mc F(\mc K(M)) = \mc K(M)^{U(\g, E)_{>0}}\cong M^{\widetilde{\mf m}_\chi} = {\mc G}(M).$$
 Therefore, the functors {$ \mc F\circ \mc K$} and $\mc G$ are isomorphic and thus $\mc K\circ M^{\theta, E}\cong \mc M^{\theta, E}$. In particular, for any $\la,\mu\in \Lambda(\nu)$, the standard Whittaker modules of $\g$ from Subsection~\ref{sect::clsofsimpleWhi} correspond to the Verma modules over $U(\g,E)$ under $\mc K$:
 \[{\mc M}^{\theta,E}(\la)\cong {\mc K}( M^{\theta, E}({\rm Wh}^{\mf l}L_{\mf l}(\la,\zeta)))\cong  \mc K(M^{\mf p}(L_{\mf l}(\la,\zeta)))\cong \mc K (M(\la,\zeta)). \]
 Therefore, the assertions in (1) follow by Lemma \ref{eq::cathatW}. By Theorem \ref{cls::thm}, all simple objects are quotients of standard Whittaker modules of $\g$. This, together with Lemma~\ref{lem::22}, implies the conclusion in (2). Finally, the assertion in (3) is a consequence of Theorem~\ref{eq::compostadWhi}. This completes the proof.
\end{proof}

\end{document}